\documentclass[twoside,11pt]{article}
\usepackage[utf8]{inputenc}
\usepackage[T1]{fontenc}
\usepackage{setspace}
\usepackage{comment}
\usepackage{fancyhdr}
\usepackage{mathtools}

\usepackage[paper=a4paper,left=30mm,right=30mm,top=30mm,bottom=30mm]{geometry}
\usepackage{tikz}
\usepackage{tikz-cd}
\usepackage{amsmath,times}
\usepackage{amsmath}
\usepackage{amssymb}
\usetikzlibrary{matrix,arrows}
\usepackage{float}
\usepackage{amsthm}
\usepackage{enumitem}
\usepackage{authblk}
\usepackage{xcolor}

\usepackage{tikz, tikz-3dplot, pgfplots}
\usepackage{tikz-cd}
\usetikzlibrary{calc}
\usetikzlibrary{positioning, calc, arrows, decorations.markings, decorations.pathreplacing}

\usepackage{hyperref}

\newtheorem{theorem}{Proposition}[section]
\newtheorem{T}[theorem]{Theorem}
\newtheorem{lemma}[theorem]{Lemma}
\newtheorem{cor}[theorem]{Corollary}

\theoremstyle{definition}
\newtheorem*{T*}{Theorem}
\newtheorem{defi}[theorem]{Definition}

\newtheorem{example}[theorem]{Example}
\newtheorem*{claim}{Claim}

\newtheorem{remark}[theorem]{Remark}

\newcommand{\N}{\mathbb{N}}
\newcommand{\Z}{\mathbb{Z}}
\newcommand{\R}{\mathbb{R}}

\DeclareMathOperator{\cl}{cl}
\DeclareMathOperator{\scl}{scl}

\DeclareMathOperator{\Aut}{Aut}
\DeclareMathOperator{\Out}{Out}
\DeclareMathOperator{\Inn}{Inn}

\DeclareMathOperator{\Acode}{\mbox{$A$}-code}
\DeclareMathOperator{\Bcode}{\mbox{$B$}-code}
\DeclareMathOperator{\Ccode}{\mbox{$C$}-code}
\DeclareMathOperator{\Zcode}{\mathbb{Z}-code}

\title{AUT-INVARIANT QUASIMORPHISMS ON GRAPH PRODUCTS OF ABELIAN GROUPS}
\date{%
}
\author{BASTIEN KARLHOFER}

\begin{document}
\maketitle

\begin{abstract}
The present paper constructs unbounded quasimorphisms that are invariant under all automorphisms on free products of more than two factors and on graph products of finitely generated abelian groups. This includes many classes of right angled Artin and right angled Coxeter groups. We discuss various geometrically arising families of graphs as examples and deduce the non-triviality of an invariant analogue of stable commutator length recently introduced by Kawasaki and Kimura for these groups. 
\end{abstract}

\section{Introduction}

The study of quasimorphisms on a given group $G$ forms a central topic in geometric group theory with quasimorphisms encoding rich information about the underlying structure of the group $G$. A classical example is the relationship of quasimorphisms with the bounded cohomology of $G$ via Bavard's duality theorem in \cite{Bavard}. For free groups $F_n$ the so called counting quasimorphisms originally introduced by Brooks in \cite{Brooks} provide a wide variety of examples. In \cite{Fuji} Calegari and Fujiwara developed his ideas further by constructing quasimorphisms on non-elementary hyperbolic groups. Many important constructions of quasimorphism for diffeomorphism groups of surfaces are carried out in \cite{Gamb}. The geometry of quasimorphisms and central extensions is analysed in  \cite{Barge}. Moreover, numerous applications of quasimorphisms in symplectic geometry originate from work of Entov and Polterovich \cite{Entov}.

In this paper we construct unbounded Aut-invariant quasimorphisms on graph products of finitely generated abelian groups answering questions on their existence stated by Marcinkowski in \cite{Marcinkowski} in almost all cases. A graph product of groups $G$ is defined by taking the free product over a set of groups indexed by the vertices of an underlying graph $\Gamma$ and introducing commutator relations for any free factors for which the according vertices are connected by an edge in $\Gamma$. Graph products of groups first rose to prominence by work of Green in her PhD thesis \cite{Green} and generalise the notions of free products and cartesian products at the same time. Right angled Artin and right angled Coxeter groups are examples of graph products where all vertex groups are either infinite cyclic or of order two respectively.   

To construct Aut-invariant quasimorphisms we first extend a result of \cite{ich} on free products of two factors, which is due to \cite{BraMarc} for the free group of rank two, to free products of more than two factors. 

\begin{T} \label{T3}
Let $G=G_1 \ast \dots \ast G_k$ be a free product of freely indecomposable groups $G_i$ where $k \geq 2$. Assume that at most two factors are infinite cyclic and there exists $j \in \{1, \dots , k \}$ such that $G_j \ncong \Z/2$. Then $G$ admits infinitely many linearly independent homogeneous Aut-invariant quasimorphisms, all of which vanish on single letters. 
\end{T}

We proceed by introducing an equivalence relation $\sim_\tau$ on the vertex set of the underlying graph of a graph product. The notion of lower cones from \cite{Marcinkowski} and an explicit description of the automorphism group of graph products of finitely generated abelian groups given in \cite{Gutierrez} enables us to construct unbounded quasimorphisms that are invariant under a finite index subgroup $\Aut^0(G) \leq \Aut(G)$, where $G$ is a graph product of finitely generated abelian groups. An averaging procedure will then produce $\Aut(G)$-invariant quasimorphisms that are still unbounded. For the case of right angled Artin groups we prove the existence of unbounded Aut-invariant quasimorphisms for many classes of graphs in Proposition \ref{RAAG1} from which we obtain the following theorem as a special case.

\begin{T}\label{RAAG2}
Let $\Gamma=(V,E)$ be a finite graph with $|V| \geq 2$ and such that no two distinct vertices $x,y \in V$ satisfy $lk(v) \subset st(v)$. Then the right angled Artin group $R_\Gamma$ admits infinitely many linearly independent homogeneous Aut-invariant quasimorphisms. 
\end{T}

Our construction yields unbounded Aut-invariant quasimorphisms for a very large class of graph products of finite abelian groups in Theorem \ref{T5}. From this we deduce the following theorem as a special case of Corollary \ref{lkcorollary}. 

\begin{T}\label{introcorollary}
Let $\Gamma=(V,E)$ be a finite graph in which no two vertices have the same link. Then either $\Gamma$ is a complete graph in which case all graph products of finite groups on $\Gamma$ are finite or any graph product of finite abelian groups on $\Gamma$ admits infinitely many linearly independent homogeneous Aut-invariant quasimorphisms. 
\end{T}

We denote the Aut-invariant stable commutator length which was recently introduced by Kawasaki and Kimura in \cite{KK} by $\scl_{\Aut}$. In Section 8, as an application of our construction, we prove the non-triviality of this invariant for many graph products among which are the ones considered in Theorem \ref{RAAG2} and Theorem \ref{introcorollary} above.

\section{Preliminaries}

\begin{defi}
Let $G$ be a group. Denote by $\Aut(G)$ the group of all automorphisms of $G$. Furthermore, denote the normal subgroup of inner automorphisms by $\Inn(G)$ and the group of outer automorphisms by $\Out(G)=\Aut(G)/\Inn(G)$.
\end{defi}

\begin{defi}
Let $G$ be a group. A map $\psi \colon G \rightarrow \R$ is called a \textit{quasimorphism} if there exists a constant $D \geq 0$ such that 
\begin{align*}
|\psi(g)+\psi(h)-\psi(gh)| \leq D \hspace{1mm} \text{for all} \hspace{1mm} g,h \in G.
\end{align*}
The \textit{defect} of $\psi$ is defined to be the smallest number $D(\psi)$ with the above property.
A quasimorphism is \textit{homogeneous} if it satisfies $\psi(g^n)=n \psi(g)$ for all $g \in G$ and all $n \in \Z$. Further, $\psi$ is called Aut-invariant if $\psi(\varphi(g))=\psi(g)$ for all $g \in G$, $\varphi \in \Aut(G)$.
\end{defi}

It is an elementary calculation to verify that homogeneous quasimorphisms are constant on conjugacy classes and vanish on elements of finite order.

\begin{defi}
Let $\psi \colon G \to \R$ be a quasimorphism. The \textit{homogenisation} $\bar{\psi} \colon G \to \R$ of $\psi$ is defined to be $\bar{\psi}(g)= \lim_{n \in \N} \frac{\psi(g^n)}{n}$ for all $g \in G$.
\end{defi}

\begin{lemma}[{\cite[p.18]{Calegari}}] \label{HomogenisationQM}
The homogenisation $\bar{\psi}$ of a quasimorphism $\psi \colon G \to \R$ is a homogeneous quasimorphism. Moreover, it satisfies $|\bar{\psi}(g)- \psi(g)| \leq D(\psi)$ for any $g \in G$. 
\end{lemma}

\begin{example}\label{Dinfty}
Let $D_\infty$ be the infinite dihedral group. Then $(D_\infty)^k$ does not admit any unbounded quasimorphism for $k \geq 1$. By Lemma \ref{HomogenisationQM} any quasimorphism is only a bounded distance away from its homogenisation. In $D_\infty$ the only elements of infinite order are conjugate to their inverses, but homogeneous quasimorphisms are constant on conjugacy classes. Thus a homogeneous quasimorphism on $D_\infty$ vanishes on all elements of infinite order and therefore on all of $D_\infty$. The same follows for $(D_\infty)^k$ for any $k \geq 1$. 
\end{example}

\begin{example}
Let $G$ be an abelian group. Then $G$ does not admit an unbounded Aut-invariant quasimorphism. In fact, the inversion $ \iota \colon G \to G$ defined by $\iota(g)=g^{-1}$ is an automorphism since $G$ is abelian. Let $\psi \colon G \to \R$ be an Aut-invariant quasimorphism. By definition $\bar{\psi}$ is Aut-invariant as well and by Lemma \ref{HomogenisationQM} $\bar{\psi}$ is only a finite distance away from $\psi$. However, $\bar{\psi}(g) = \bar{\psi}(\iota(g)) = \bar{\psi}(g^{-1}) = - \bar{\psi}(g)$ for all $g \in G$. So $\bar{\psi}$ vanishes and $\psi$ was bounded to begin with. 
\end{example}

\begin{defi}
Let $I$ be an indexing set and let $G= \ast_{i \in I} G_i$ be a free product of a family of groups $\{G_i\}_{i \in I}$. Via the canonical inclusion the factor $G_i$ is a subgroup of $G$ for each $i \in I$. An element of $G$ belonging to one of the factors is called a \textit{letter} of $G$. A \textit{word} in $G$ is a product of letters in $G$. For any two letters belonging to the same factor in $G$ their product in $G$ can be replaced by the letter that represents their product in that factor. Moreover, any letters that are the identity in the factor they belong to can be omitted inside any word without changing the element that word represents in $G$. A word is called \textit{reduced} if no two consecutive letters belong to the same factor and no letters appear that represent the identity. Recall that every element $g \in G$ has a unique presentation as a reduced word.
\end{defi}

\begin{defi}
A group $G$ is called \textit{freely indecomposable} if $G$ is non-trivial and not isomorphic to a free product of the form $G_1 \ast G_2$ where $G_1$, $G_2$ are non-trivial groups.
\end{defi}
 
For example, every finite group is freely indecomposable since any free product of non-trivial groups contains elements of infinite order.
Every group with non-trivial center is freely indecomposable since any free product of non-trivial groups has trivial center.

\subsection{New quasimorphisms from old ones}

\begin{lemma} \label{Gen6}
Let $G$ be a group and let $H \leq G$ be a characteristic subgroup with associated quotient map $p \colon G \rightarrow G/H$. Then for any unbounded Aut-invariant quasimorphism $\psi \colon G/H \to \R$ the composition $\psi \circ p \colon G \to \R$ is an unbounded Aut-invariant quasimorphism on $G$. Furthermore, linearly independent quasimorphisms on $G/H$ give rise to linearly independent quasimorphisms on $G$. 
\end{lemma}

\begin{proof}
Since $p$ is a homomorphism, $\psi \circ p$ is a quasimorphism. The Aut-invariance of $\psi \circ p$ on $G$ follows from the Aut-invariance of $\psi$ on $G/H$ together with the fact that $H$ is characteristic. Finally, the statement about linear independence follows from the surjectivity of $p$. 
\end{proof}

\begin{lemma}\label{I=J}
Let $H \leq G$ be a subgroup of finite index $k \in \N$ and $I=\{g_1, \dots, g_k \}$ be a set of right coset representatives. So $G = \bigcup_{i=1}^n Hg_i$ where the union is disjoint. Let $g \in G$ be arbitrary. For each $i \in \{1, \dots, k\}$ we can write $g_i g =h_{i} g_{j_i}$ uniquely where $j_i \in \{1, \dots , k \}$ and $h_{i} \in H$. Define \begin{align*}
J_g=\{g_{j_i} \in I \ |\ \exists i \in \{1, \dots, k\} \text{ such that } g_i g =h_{i} g_{j_i} \text{ where } h_{i} \in H  \}.
\end{align*}
Then $J_g=I$. 
\end{lemma}

\begin{proof}
Let $g \in G$. Since $I$ is a system of coset representatives, the product of $g_i$ with $g$ can indeed be written for all $i \in \{1, \dots ,k \}$  as $g_i g = h_{i}g_{j_i}$ where $h_{i} \in H$ and $j_i \in \{1, \dots, k \}$. Then 
\[
G=Gg= \bigcup_{i=1}^n Hg_ig = \bigcup_{i=1}^n H h_{i}g_{j_i} = \bigcup_{i=1}^n H g_{j_i}
\]
where the union is disjoint. Therefore, the set $J_g$ consisting of all $g_{j_i}$ appearing in this disjoint union needs to contain all coset representatives. That is $J_g=I$ independently of the choice of $g \in G$. 
\end{proof}

\begin{lemma}\label{finindexinvariant}
Let $\psi \colon G \to \R$ be a quasimorphism invariant under a subgroup $H \leq \Aut(G)$ of index $k \in \N$. Let $\{f_1, \dots, f_k \}$ be a set of right coset representatives. Let $\widehat{\psi} \colon G \to \R$ be defined by $\widehat{\psi}(g) = \sum_{i=1}^k \psi(f_i(g))$ for all $g \in G$. Then $\widehat{\psi}$ is an $Aut$-invariant quasimorphism on $G$.
\end{lemma}

\begin{proof}
Clearly, $\widehat{\psi}$ is a quasimorphism since it is a finite sum of quasimorphisms. Let $\theta \in Aut(G)$ and $g \in G$. For all $i$ we can uniquely write $f_i \circ \theta= h_{i} \circ f_{j_i}$ where $j_i \in \{1, \dots, k\}$ and $h_{i} \in H$ since $\{f_1, \dots, f_k \}$ is a set of right coset representatives of $H$ in $\Aut(G)$. We calculate
\begin{align*}
\widehat{\psi}(\theta(g))  &= \sum_{i=1}^k \psi(f_i(\theta(g)))  
 = \sum_{i=1}^k \psi((f_i \circ \theta)(g)) 
 = \sum_{i=1}^k \psi((h_{i} \circ f_{j_i})(g)) \\
 &\overset{\mathclap{\strut\text{H-invariance of } \psi}}=  \hspace{8mm}
 \sum_{i=1}^k \psi( f_{j_i}(g)) 
 \hspace{6mm} \overset{\mathclap{\strut\text{Lemma \ref{I=J}}}}= \hspace{6mm} \sum_{i=1}^k  \psi(f_i(g))
 = \widehat{\psi}(g).
\end{align*}
The penultimate equality follows from the fact that due to Lemma \ref{I=J} the set of all $f_{j_i}$ appearing in $\sum_{i=1}^k \psi( f_{j_i}(g))$ agrees with the set of all $f_i$ appearing in $\sum_{i=1}^k  \psi(f_i(g))$. 
\end{proof}

\section{Code quasimorphisms}

\begin{remark}
This entire section on the construction of (weighted) code quasimorphisms is only required for the last step of the proof of Theorem \ref{T5} where the naturality statement of Lemma \ref{naturalitycode} is used. All other results can be deduced purely from Theorem \ref{T1} without any knowledge of the explicit construction of Aut-invariant quasimorphisms on a free product of two freely indecomposable groups. 
\end{remark}

Recall that a \textit{tuple} always refers to a finite sequence and therefore all tuples are naturally ordered.

\begin{defi}[Code]
Let $A$ and $B$ be groups. Write $g \in A \ast B$ in its reduced form. Assign two tuples of non-zero natural numbers that we will call \textit{codes} to $g$ as follows. The tuple  $( a_1, \dots, a_k)$ of letters from $A$ appearing in the reduced form of $g$ is called the \textit{A-tuple} of $g$. We then count how often any one letter of $( a_1, \dots, a_k)$ appears consecutively which yields a tuple of positive numbers $\Acode(g)=(n_1, n_2, \dots, n_r)$ called the \textit{A-code} of $g$. In the same way we obtain the \textit{B-tuple} of $g$ as the tuple of letters from $B$ appearing in the reduced form of $g$ and the \textit{B-code} of $g$, denoted $\Bcode(g)$, by counting consecutive appearances of letters in the \textit{B-tuple}.  

%In the same way we obtain a tuple of natural numbers $\Bcode(g)=(m_1, m_2, \dots, m_s)$, called the \textit{B-code} of $g$, coming from counting consecutive appearances in the \textit{B-tuple} $( b_1, \dots , b_\ell)$ of letters from $B$. 
\end{defi}

\textit{Code quasimorphisms} are defined in the spirit of Brooks's counting quasimorphisms and are counting the occurrences of a tuple of natural numbers in the $\Acode$ and $\Bcode$ associated to an element in the free product $A \ast B$. 

\begin{defi}[Code quasimorphisms]
Let $k \geq 1$ and let $z=(n_1, \dots ,n_k)$ be a tuple of non-zero natural numbers $n_1, \dots ,n_k$. Let $C \in \{ A, B \}$. Define
$\theta^C_z \colon A \ast B \to \Z_{\geq 0}$ to count the maximal number of \textit{disjoint} occurrences of $z$ as a tuple of consecutive numbers in the $\Ccode$ for all $g \in A \ast B$. Further, define the \textit{code quasimorphism} 
\begin{align*}
f^C_z \colon A \ast B \to \Z \hspace{5mm} \text{by} \hspace{5mm} f^C_z(g)= \theta^C_{z}(g) - \theta^C_{z}(g^{-1})
\end{align*}
for all $g \in A \ast B$. It is shown in \cite[Lemma 4.8]{ich} that code quasimorphisms are indeed quasimorphisms. By definition any code quasimorphism is bounded by 1 on all letters. 
\end{defi}

\begin{example}
Let $G=A \ast B$ for $A= \Z/5$ and $B \neq 1$ be freely indecomposable. Consider  
\[
g=a^4ba^2ba^2ba^3bababa^3bababa^2ba^2ba^2b
\]
for non-trivial $a \in A, b \in B$. Then the A-tuple of $g$ is $( a^4,a^2,a^2,a^3,a,a,a^3, a, a, a^2,a^2,a^2)$. It follows that $\Acode(g)=(1,2,1,2,1, 2,3)$. For $z=(1,2,1)$ we calculate that $\theta^A_z(g)=1$ since we only count disjoint occurrences. Note that $\Acode(g^{-1})=(3,2,1,2,1,2,1)$. So $\theta^A_{z}(g^{-1})=1$ and $f^A_z(g)=0$. However, for $w=(1,2,3)$ we have $\theta^A_w(g)=1$ and $\theta^A_{w}(g^{-1})=0$. So $f^A_w(g)=1$. In fact, for all $n \in \N$ we have $\theta^A_w(g^n)=n$, $\theta^A_{w}(g^{-n})=0$ and so $f^A_w(g^n)=n$. This implies that the homogenisation $\bar{f}^A_w$, whose Aut-invariance for $B \notin \{ \Z/2, \Z \}$ will be established by Proposition \ref{freeprod no Z prop} , satisfies $\bar{f}^A_w(g) > 0$ and is unbounded on powers of $g$. 
\end{example}

\begin{defi}
A tuple of non-zero natural numbers $z=(n_1, \dots, n_k)$  is defined to be \textit{generic} if $\bar{z}$ does not occur as a tuple of $k$ adjacent numbers in $z^2=(n_1, \dots, n_k, n_1, \dots n_k)$. 
\end{defi}

\begin{example}
Let $z=(n_1, \dots, n_k)$. If $k \leq 2$, $z$ is not generic. If $k \geq 3$ and the $n_i$ are pairwise distinct, then $z$ is generic. E.g. for $z=(1,2,3)$ we have $\bar{z}=(3,2,1)$ does not appear in $z^2=(1,2,3,1,2,3)$. 
\end{example}

\begin{theorem}\cite[Prop. 4.11]{ich} \label{freeprod no Z prop}
Let $A \ast B$ be a free product of two freely indecomposable groups $A$ and $B$, neither of which is infinite cyclic. Then for any generic tuple of natural numbers $z$ the following holds:
\begin{enumerate}
\item if $A \ncong B$ and $C \in \{A,B \}$ is such that $C \ncong \Z/2$, then the homogenisation $\bar{f}^C_z$ of the quasimorphism $f^C_z$ is an unbounded Aut-invariant quasimorphism on $A \ast B$;
\item if $A \cong B \ncong \Z/2$, then the sum $\bar{f}^A_z + \bar{f}^B_z$ is an unbounded Aut-invariant quasimorphism on $A \ast B$. 
\end{enumerate}
In both cases the space of homogeneous Aut-invariant quasimorphisms on $A \ast B$ has infinite dimension.
\end{theorem}

\subsection{Weighted code quasimorphisms}

\begin{defi}[Weighted $\Z$-code]
Let $B$ be a group which is freely indecomposable and not infinite cyclic. Write $g \in \Z \ast B$ in reduced form. Let $(a_1, \dots , a_k)$ be the $\Z$-tuple of $g$. Associate a tuple $(x_1, \dots , x_\ell)$ of non-zero natural numbers to $g$ as follows. Consider the successive subsequences of maximal length in $(a_1, \dots ,a_k)$ consisting of integers all of the same sign. For the $i$-th such sequence, we define $x_i$ to be the absolute value of the sum of integers in that sequence. The tuple $(x_1, \dots , x_\ell)$ is called the \textit{weighted $\Z$-code} of $g$.
\end{defi}

\begin{example}
Let $B$ be a non-trivial group and let $b_i \in B$ be non-trivial elements. Then the reduced word 
\begin{align*}
w=8b_1 (-4) b_2 (-4) b_3 (-1) b_4 7b_5 2 b_6 (-3)
\end{align*} 
has the $\Z$-tuple $(8,-4.-4,-1,7,2,-3)$ which yields the weighted $\Zcode$ $(8,9,9,3)$.
\end{example}

\begin{defi}[Weighted code quasimorphisms]
Let $z=(n_1, \dots , n_k)$ be a tuple of non-zero natural numbers. Define $\theta^\Z_z \colon \Z \ast B \to \Z_{\geq 0}$ to count the number of \textit{disjoint} occurrences of $z$ as a tuple of consecutive numbers inside the weighted $\Zcode$ of $g \in \Z \ast B$. For $g \in \Z \ast B$ define the \textit{weighted code quasimorphism}  
\begin{align*}
f^\Z_z \colon \Z \ast B \to \Z \hspace{5mm} \text{by} \hspace{5mm} f^\Z_z(g)= \theta^{\Z}_{z}(g) - \theta^\Z_{z}(g^{-1}).
\end{align*}
\end{defi}

\begin{theorem}\cite[Prop. 5.7]{ich} \label{freeprod with Z}
Let $B$ be a freely indecomposable group which is not infinite cyclic. Then for any generic tuple of natural numbers $z$ the homogenisation $\bar{f}^\Z_z \colon \Z \ast B \to \R$ of the quasimorphism $f^\Z_z$ is an unbounded Aut-invariant quasimorphism on $\Z \ast B$. Moreover, the space of homogeneous Aut-invariant quasimorphisms on $\Z \ast B$ that vanish on letters has infinite dimension. 
\end{theorem}

\subsection{Naturality of (weighted) code quasimorphisms with respect to inclusions}

\begin{lemma}\label{naturalitycode}
Let $A,B,C,D$ be freely indecomposable groups such that $ C \leq A$, $D \leq B$ are subgroups. Let $G=A \ast B $ and let $H=C \ast D$ inside $G$. Let $z$ be any tuple of positve integers and consider the counting functions $\theta^A_z \colon G \to \Z$ and $\theta^C_z \colon H \to \Z$. Then $(\theta^A_z)|_{H} = \theta^C_z$. Consequently, we have $(f^A_z)|_{H} = f^C_z$ for the corresponding code quasimorphisms which implies $(\bar{f}^A_z)|_{H} = \bar{f}^C_z$ for their homogenisations.
\end{lemma}

\begin{proof}
Any word in $H$ is reduced if and only if its image in $G$ is reduced. Therefore, for any $h \in H$ with $C$-tuple given by $(c_1, \dots, c_k)$ where $c_i \in C$ for all $i$ we have that the $A$-tuple of $h$ is given by $(c_1, \dots, c_k)$ as well. Consequently, we have $\Acode(h) = \Ccode(h)$ for all $h \in H$. So $(\theta^A_z)|_{H} = \theta^C_z$ for all tuples $z$ of positive integers and the rest of the statement follows as well. 
\end{proof}

\section{Graph products of abelian groups}

\subsection{Definitions}

\begin{defi}[Graph]
A finite graph $\Gamma$ is a pair $(V,E)$ consisting of a non-empty finite set of vertices $V$ and a finite set of unoriented edges $E$ where an unoriented edge is a subset of $V$ of cardinality two. So two distinct vertices have at most one edge between them and no vertex has an edge to itself. A graph $\Gamma$ is called \textit{complete} if there exists an edge between all pairs of distinct vertices in $\Gamma$. For a subset $X \subset V$ we denote by $\Gamma_X$ the subgraph of $\Gamma$ spanned by $X$. 
\end{defi}

\begin{defi}
Let $\Gamma=(V,E)$ be a graph.
\begin{enumerate}
\item The link of $v \in V$ is defined to be $lk(v)= \{ x \in V : (v,x) \in E\}$.
\item The star of $v \in V$ is defined to be $st(v) = lk(v) \cup \{v \}$. 
\end{enumerate}
\end{defi}

\begin{defi}[{\cite[Def. 5.4]{Gutierrez}}]
A preorder is a relation that is reflexive and transitive. We define two preorders on any graph $\Gamma=(V,E)$ by: 
\begin{align*}
v \leq w & \hspace{2mm} \text{ if and only if } \hspace{2mm} lk(v) \subset st(w),  \\
v \leq_s w & \hspace{2mm} \text{ if and only if } \hspace{2mm} st(v) \subset st(w). 
\end{align*} 
\end{defi}

The preorder $\leq $ has already been defined before in \cite{Charney} with the proof of transitivity given in \cite[Lemma 2.1]{Charney}.

\begin{defi}[Graph product]
Let $\Gamma = (V,E)$ be a finite graph. Let $\{G_v\}_{v \in V}$ be a set of non-trivial groups. The graph product defined by $\Gamma$ and $\{G_v \}_{v \in V}$ is the group 
\begin{align*}
W(\Gamma, \{G_v\}_{v \in V})=(\ast_{v \in V} G_v)/N,
\end{align*}
where $N$ is the normal subgroup generated by all $[G_x, G_y]$ for $x,y \in V$ such that $\{x,y\} \in E$.  
\end{defi}

\begin{defi}
Let $W(\Gamma, \{G_v\}_{v \in V})$ be a graph product of groups. If all vertex groups are $\Z/2$, then $W(\Gamma, \{G_v\}_{v \in V})$ is the \textit{right angled Coxeter group} on the graph $\Gamma$. If all vertex groups are infinite cyclic, then $W(\Gamma, \{G_v\}_{v \in V})$ is the \textit{right angled Artin group} on $\Gamma$ which will be denoted by $R_\Gamma$ from now onwards.
\end{defi}

\begin{defi}[Truncated subgroup]
Let $\Gamma=(V,E)$ be a graph and $\{G_v\}_{v \in V}$ be a collection of finitely generated abelian groups. Let $V^\prime \subset V$ be a non-empty subset of vertices and let $\Gamma^\prime$ be the subgraph of $\Gamma$ spanned by $V^\prime$. Then $W(\Gamma^\prime, \{G_v \}_{v \in V^\prime})$ is called a \textit{truncated subgroup} of $W(\Gamma, \{G_v \}_{v \in V})$. If $W(\Gamma, \{G_v \}_{v \in V})$ decomposes as the cartesian product of the truncated subgroup $W(\Gamma^\prime, \{G_v \}_{v \in V^\prime})$ spanned by $V^\prime \subset V$ and the one spanned by $V-V^\prime$, then $W(\Gamma^\prime, \{G_v \}_{v \in V^\prime})$ is called a \textit{direct truncated subgroup} of $W(\Gamma, \{G_v \}_{v \in V})$.
\end{defi}

A group that is isomorphic to $\Z/p^k$ where $k$ is a positive integer and $p$ is a prime is called \textit{primary}. Recall that finitely generated abelian groups can be uniquely decomposed into a cartesian product with primary and infinite cyclic factors. Thus, one sees that for any graph product of finitely generated abelian groups $W(\Gamma, \{G_v\}_{v \in V})$ there is a graph $\Gamma^\prime=(V^\prime ,E^\prime)$ together with a collection of groups $\{G_v^\prime \}_{v \in V^\prime}$, where each $G_v^\prime$ is primary or infinite cyclic, such that $W(\Gamma, \{G_v\}_{v \in V}) \cong W(\Gamma^\prime, \{G_v^\prime \}_{v \in V^\prime})$. Namely, $\Gamma^\prime=(V^\prime,E^\prime)$ is obtained by replacing each $v \in V$ by the complete graph with vertices corresponding to the generators of the primary and infinite cyclic factors of $G_v$. We say that $\Gamma^\prime$ is the graph \textit{expanding} $\Gamma$ in this case. Clearly, for a graph product of finite abelian groups $W(\Gamma, \{G_v\}_{v \in V})$ there is a collection $\{G_v^\prime \}_{v \in V^\prime}$ consisting only of primary groups such that $W(\Gamma, \{G_v\}_{v \in V}) \cong W(\Gamma, \{G_v^\prime \}_{v \in V^\prime})$. 

We simplify notation by denoting $ W(\Gamma, \{G_v \}_{v \in V})$ as $W_\Gamma$ for the case where the set $\{G_v \}_{v \in V}$ only consists of primary or infinite cyclic groups. The next lemma shows that the above replacement of $W(\Gamma, \{G_v\}_{v \in V})$ by  $W(\Gamma, \{G_v^\prime \}_{v \in V^\prime})$ does not create new direct truncated subgroups that are free products of primary or infinite cyclic groups.

\begin{lemma}\label{direct trunc subgroup}
Let $\Gamma=(V,E)$ be a graph and let $W(\Gamma, \{G_v\}_{v \in V})$ be a graph product of finitely generated abelian groups. Let $\Gamma^\prime = (V^\prime , E^\prime)$ and $\{G_v^\prime\}_{v \in V^\prime}$ where each $G_v^\prime$ is primary or infinite cyclic be the graph expanding $W(\Gamma, \{G_v\}_{v \in V})$ such that $W(\Gamma^\prime, \{G_v^\prime\}_{v \in V^\prime} \cong W(\Gamma, \{G_v\}_{v \in V})$. Assume that for some $k \geq 2$  there is a direct truncated subgroup in $W(\Gamma^\prime, \{G_v^\prime\}_{v \in V^\prime})$ of the form $H=G_{v_1} \ast \dots \ast G_{v_k}$, then $H$ is a direct truncated subgroup of $W(\Gamma, \{G_v\}_{v \in V})$ generated by the same vertex set.
\end{lemma}

\begin{proof}
For $v_1 \in V^\prime$ there is a corresponding vertex $w_1 \in V$ such that $v_1$ originates from $w_1$ when replacing the finitely generated abelian vertex group $G_{w_1}$ by its primary and infinite cyclic factors. Assume there was another generator $w$ inside the vertex group $G_{w_1}$ yielding another vertex $w$ in the graph $\Gamma^\prime$. Then  since $v_1$ and $w$ originate from the same vertex group they commute in $W(\Gamma^\prime, \{G_v^\prime\}_{v \in V^\prime}$ as well and so $w \notin\{v_1, \dots, v_k\}$. Then $w$ belongs to the truncated subgroup spanned by $V^\prime - \{v_1, \dots, v_k \}$. Since $H$ is a direct truncated subgroup, this means that $w$ commutes with all $v_i$. However, since $v_1$ and $w$ originate from the same vertex group, they also satisfy $st(w)=st(v_1)$. This is a contradiction since $v_2 \in st(w)$ but $v_2 \notin st(v_1)$. Therefore, the vertex group $G_{w_1}$ is the same as $G_{v_1}$ to begin with. Similarly, this holds for all $v_i$ for $i \geq 2$ and $\{ w_1, \dots, w_k \}$ span $H$ as a direct truncated subgroup in $W(\Gamma, \{G_v\}_{v \in V})$.   
\end{proof}

The question of when a graph product decomposes as a direct product of truncated subgroups has been solved in \cite{Green} for graph products with general vertex groups. We only give a version restricted to our case of finitely generated abelian vertex groups here, where we assume that the vertex set of the graph in consideration has cardinality at least two since truncated subgroups have non-trivial underlying vertex set by definition.

\begin{T}[{\cite[Thm. 3.34]{Green}}]\label{Green1}
Let $\Gamma=(V,E)$ be a graph with $|V| \geq 2$ and let $\{G_v \}_{v \in V}$ be a collection of finitely generated abelian groups. Then $W(\Gamma, \{G_v \}_{v \in V})$ has non-trivial center if and only if  $W(\Gamma, \{G_v \}_{v \in V})$ is a direct product of truncated subgroups at least one of which has non-trivial center. 
\end{T}

Inductively applying this result immediately yields the following description of graph products of finitely generated abelian groups. 

\begin{cor}\label{Green2}
Let $\Gamma$ be a graph and $\{G_v \}_{v \in V}$ be a collection of finitely generated abelian groups. Then $W(\Gamma, \{G_v \}_{v \in V})$ has non-trivial center if and only if   there exists a vertex $v \in V$ such that $st(v)=V$. 
\end{cor}

\begin{proof}
Clearly, any $v \in V$ such that $st(v)=V$ belongs to the center of $W(\Gamma, \{G_v \}_{v \in V})$. Conversely, if the center of $W(\Gamma, \{G_v \}_{v \in V})$ is non-trivial, then it decomposes into a direct product of truncated subgroups at least one of which has a non-trivial center. Since truncated subgroups are graph products themselves , we can iterated this procedure until we end up with a subgroup over a single vertex $v \in V$ that is a direct factor of  $W(\Gamma, \{G_v \}_{v \in V})$. This implies $st(v)=V$. 
\end{proof}

\subsection{Automorphisms of graph products}

There are the following four families of automorphisms for graph products $W_\Gamma$ for which each vertex group is primary or infinite cyclic (see  \cite[p.1--2]{Gutierrez} combined with {\cite[Prop. 5.5]{Gutierrez}}). From now onwards we will abuse notation to identify each vertex of $\Gamma$ with a chosen generator of the cyclic group $G_v$. Therefore, vertices can be considered as elements in $W_\Gamma$.  
\begin{enumerate}
\item Every isomorphism of graphs $\gamma \colon \Gamma \to \Gamma$ such that $G_v=G_{\gamma(v)}$ for all $v \in V$ induces an automorphism of $W_\Gamma$. Such automorphisms are called \textit{labelled graph automorphisms}. 
\item Let $v \in V$ and $m \in \Z$ such that $\gcd(m, |G_v|)=1$. If $|G_v|=\infty$, then $m= \pm 1$. Then we set the \textit{factor automorphism} $\phi_{v,m}$ to be the unique automorphism of $W_\Gamma$ defined for $z \in V$ by
\begin{align*}
\phi_{v,m}(z)=
\begin{cases}
 v^m & z=v, \\
 z & z \neq v.
\end{cases}
\end{align*}
\item Let $v,w \in V$ be two distinct vertices. A \textit{dominated transvection} is an automorphism $\tau_{v,w}$ defined having one of the two forms: 
\begin{enumerate}
\item $|G_v|= \infty$, $v \leq w$ and 
\begin{align*}
\tau_{v,w}(z)=
\begin{cases}
vw & z=v, \\
z & z \neq v,
\end{cases}
\end{align*}
\item $|G_v|=p^k$, $|G_w|=p^\ell$, $v \leq_s w$ and 
\begin{align*}
\tau_{v,w}(z)=
\begin{cases}
vw^q & z=v, \\
z & z \neq v,
\end{cases}
\end{align*}
where $q = \max \{ 1, p^{\ell -k} \}$ and $p$ is prime. 
\end{enumerate}
\item Let $v \in V$ and let $K$ be the vertex set of a connected component of $\Gamma - st(v)$. Define the \textit{partial conjugation} by $v$ on $K$ by
\begin{align*}
\sigma_{K,v}(z)=
\begin{cases}
vzv^{-1} & z \in K, \\
z & z \notin K.
\end{cases}
\end{align*}  
\end{enumerate}

\begin{T}[{\cite[p.2]{Gutierrez}}]
Let $G=W_\Gamma$ be a graph product where each vertex group is primary or infinite cyclic. Then $\Aut(G)$ is generated by the above four families of automorphisms.
\end{T}

\begin{defi} \label{defaut0}
Let $G=W_\Gamma$. Then the subgroup of $\Aut(G)$ generated by automorphisms of type 2--4 above is denoted by $\Aut^0(G)$. 
\end{defi}

\begin{lemma}\cite[Prop 2.1]{Marcinkowski} \label{MarciLemma}
$\Aut^0(W_\Gamma)$ is a finite index subgroup of $\Aut(W_\Gamma)$. A set of coset representatives is given by labelled graph automorphisms.
\end{lemma}

\subsection{Automorphisms of free products}

Let $G= G_1 \ast \dots \ast G_k$ be a free product of freely indecomposable groups that are not necessarily finitely generated and abelian. Then there are the following classes of automorphisms of $G$ that can be thought of in the above framework where all connected components are just single points. 
\begin{enumerate}
\item The labelled graph automorphisms are generated by the so called \textit{swap automorphisms} that swap two isomorphic free factors.
\item Applying the automorphism of any factor to that factor defines an automorphism of the free product that is called a \textit{factor automorphism}.
\item Multiplying the generator $v$ of an infinite cyclic factor by any element $w$ belonging to another factor always defines a \textit{transvection} that is an automorphism. There are no transvections for factors on factors that are not infinite cyclic. 
\item \textit{Partial conjugations} are defined in the same way as for graph products where each vertex is its own connected component.
\end{enumerate}

In the case of free products of primary and infinite cyclic groups this gives the same description of the automorphisms group as in the previous section with the only exception being the more efficient description of labelled graph automorphisms by their generating set of swap automorphisms. It was established in \cite{Fouxe1} and \cite{Fouxe2} that these classes generate the automorphism group of a free product.

\begin{T}[Fouxe-Rabinovitch] \label{Fouxe}
Let $G= G_1 \ast \dots \ast G_k$ be a free product of freely indecomposable groups. Then $\Aut(G)$ is generated by the classes of swap automorphisms, factor automorphisms, transvections and partial conjugations.
\end{T}

\begin{defi}
Denote by $\Aut^0(G)$ again the subgroup of  a free product of freely indecomposable groups $G= G_1 \ast \dots \ast G_k$ generated by automorphisms of type 2--4. This agrees with Definition \ref{defaut0} if all free factors are primary or infinite cyclic groups. 
\end{defi}

In a free product $G= G_1 \ast \dots \ast G_k$ any two choices of swap automorphisms interchanging two factors $G_i$ and $G_j$ differ by a product of factor automorphisms of these factors. Since the number of permutations of a finite set is finite one immediately concludes the following lemma.

\begin{lemma} \label{BarciLemma}
Let $G= G_1 \ast \dots \ast G_k$ be a free product of freely indecomposable groups. Then $\Aut^0(G) $ has finite index in $\Aut(G)$ with a system of coset representatives given by a choice of permutations of the isomorphic free factors of $G$.
\end{lemma}

\section{General Case of free products}

The following theorem in \cite{ich} is a combination of Proposition \ref{freeprod no Z prop} , Proposition \ref{freeprod with Z} and a result on $F_2$ in \cite[Thm. 2]{BraMarc}. Together with an averaging procedure over the swap automorphisms this will enable us to prove Theorem \ref{T3} in this section afterwards. 

\begin{T}[Theorem 1 in \cite{ich}]  \label{T1}
Let $G= A \ast B$ be the free product of two non-trivial freely indecomposable groups $A$ and $B$. Assume $G$ is not the infinite dihedral group. Then $G$ admits infinitely many linearly independent homogeneous Aut-invariant quasimorphisms, all of which vanish on single letters. 
\end{T}

Recall, that for a free product $G$ the group $\Aut^0(G)$ is the subgroup of its automorphism group generated by factor automorphisms, transvections and partial conjugations.

\begin{lemma}\label{Aut0-projection}
Let $k \geq 3$ and let $G=G_1 \ast G_2 \ast \dots \ast G_k$ be a free product of freely indecomposable factors such that $G_i$ is not infinite cyclic for $i \geq 3$. Then the kernel $\ker(p)$ of the projection $p \colon G \to G_1 \ast G_2$ is $Aut^0(G)$-invariant. 
\end{lemma}

\begin{proof}
The kernel $\ker{p}$ is normally generated by letters belonging to $G_3, \dots, G_k$. Thus, $\ker(p)$ is clearly invariant under all factor automorphisms. Since $G_i$ is not infinite cyclic for $i \geq 3$, all transvections act trivially on letters of $G_i$. Moreover, since $\ker(p)$ is a normal subgroup, it is invariant under all partial conjugations. These three classes of automorphisms generate $\Aut^0(G)$ and so $\ker(p)$ is invariant under $\Aut^0(G)$.
\end{proof}

\begin{lemma} \label{precompaut}
Let $k \geq 3$ and let $G=G_1 \ast G_2 \ast \dots \ast G_k$ be a free product of freely indecomposable factors such that $G_i$ is not infinite cyclic for $i \geq 3$. Let $p \colon G \to G_1 \ast G_2$ be the projection and let $\psi \colon G_1 \ast G_2 \to \R$ be an Aut-invariant quasimorphism. Then $\psi \circ p \colon G \to  \R$ is an $Aut^0$-invariant quasimorphism on $G$. 
\end{lemma}

\begin{proof}
Clearly, $\psi \circ p$ is a quasimorphism. Its $\Aut^0$-invariance follows from Lemma \ref{Aut0-projection} together with the Aut-invariance of $\psi$.  
\end{proof}

\begin{proof}[Proof of Thm 1.1]
After reordering the free factors we may assume that $G_i$ is not infinite cyclic for $i \geq 3$ and that $G_1 \ncong \Z/2$. Let $\psi$ be an Aut-invariant quasimorphism on $G_1 \ast G_2$ that is homogeneous, unbounded and vanishes on single letters. The existence of such a quasimorphism follows from Theorem \ref{T1} . By Lemma \ref{precompaut} the composition $\psi \circ p$, where $p \colon G \to G_1 \ast G_2$ denotes the projection map, is $\Aut^0$-invariant on $G$. By Lemma \ref{BarciLemma} the index of $\Aut^0(G)$ in $\Aut(G)$ is finite with a system of coset representatives given by automorphisms permuting the factors. We denote this system by $\{\sigma_i\}_{i \in I}$. Therefore, by Lemma \ref{finindexinvariant} the quasimorphism $\widehat{\psi \circ p}$ defined for $g \in G$ by $\widehat{\psi \circ p}(g)= \sum_i \psi(p(\sigma_i(g)))$ is invariant under the whole automorphism group of $G$. It is homogeneous as well. Let us verify that it is unbounded.

Let $g \in G_1 \ast G_2$. For any $i \in I$ and $j \in \{1, \dots, k \}$ we have on the level of letters that $\sigma_i(G_j) = G_{\sigma_i(j)}$. So, if $\{\sigma_i(1),\sigma_i(2)\} \neq \{1,2 \}$ then $p(\sigma_i(g))$ is either trivial or just a single letter of $G_1$ or $G_2$. However, we know that $\psi$ vanishes on single letters and so $\psi(p(\sigma_i(g)))$ vanishes for all $\sigma_i$ that fail to satisfy $\{\sigma_i(1), \sigma_i(2) \} = \{1,2 \}$. So let $J \leq I$ be the subset consisting of $\sigma_j$ satisfying $\{\sigma_j(1), \sigma_j(2) \} = \{1,2 \}$. Any $\sigma_j \in J$ descends to an automorphism $\bar{\sigma_j}$ of $G_1 \ast G_2$ under $p$. However, we know that $\psi$ is invariant under all automorphisms of $G_1 \ast G_2$, so in particular under the ones of the form $\bar{\sigma_j}$. Consequently, we calculate for $g \in G_1 \ast G_2$ that
\begin{align*}
\widehat{\psi \circ p}(g)= \sum_{i \in I} \psi(p(\sigma_i(g))) = \sum_{j \in J} \psi(p(\sigma_j(g))) = \sum_{j \in J} \psi(\bar{\sigma_j}(g))) = |J| \psi(g), 
\end{align*}
where $|J|$ denotes the cardinality of the set $J$. Note that $|J| \geq 1 $ since it contains the representative of the class representing the identity.

This means that for any Aut-invariant quasimorphism $\psi$ on $G_1 \ast G_2$ that vanishes on letters the Aut-invariant quasimorphism $\widehat{\psi \circ p}$ on $G$ restricts to a linear multiple of $\psi$ on the subgroup $G_1 \ast G_2 \leq G$. Additionally, $\widehat{\psi \circ p}$ vanishes on all letters as well. Consequently, the space of homogeneous Aut-invariant quasimorphisms on $G$ is infinite dimensional since the space of homogeneous Aut-invariant quasimorphisms on $G_1 \ast G_2$ is infinite dimensional by Theorem \ref{T1}. Moreover, all of these homogeneous quasimorphisms can be chosen to vanish on letters.
\end{proof}

\section{Aut-invariant quasimorphisms on graph products}

From now onwards we will only consider graph products of finitely generated abelian groups.

\subsection{Lower cones}

\begin{defi}
Let $\Gamma=(V,E)$ be a graph and let $X \subset V$ be a subset of vertices. The map $R_X \colon W_\Gamma \to W_{\Gamma_X}$ defined by 
\begin{align*}
R_X(z)= 
\begin{cases}
z & z \in X, \\
e & z \notin X.
\end{cases}
\end{align*}
is called a \textit{standard retraction} onto the full subgraph of $\Gamma$ generated by $X$. We denote the kernel of $R_X$ by $K_X$. 
\end{defi}

\begin{lemma}[{\cite[Lemma 3.3]{Marcinkowski}}]
Let $X \subset V$. The group $K_X$ is invariant under factor automorphisms and partial conjugations. 
\end{lemma}

\begin{defi}\label{deforder}
Let $\Gamma=(V,E)$ be a graph and let $W(\Gamma, \{G_v \}_{v \in V})$ be a graph product of a family of groups $\{G_v \}_{v \in V}$ each of which is primary or infinite cyclic.
Let $\leq_\tau$ be the relation on $V$ defined for $v,w \in V$ by $v \leq_\tau w$ if and only if the dominated transvection $\tau_{v,w}$ is well-defined.   
\end{defi}

 Definition \ref{deforder} above is the central definition of this subsection. Thus, for the remainder of this subsection we always implicitly assume that $\Gamma=(V,E)$ is a graph and $W_\Gamma=W(\Gamma, \{G_v \}_{v \in V})$ is a graph product of a family of groups $\{G_v \}_{v \in V}$ each of which is primary or infinite cyclic.

\begin{lemma}[{\cite[Lemma 3.4]{Marcinkowski}}]
The relation $\leq_\tau$ is a preorder on $V$.
\end{lemma}

\begin{defi}
Define a relation $\sim_\tau$ on $V$ by setting $v \sim_\tau w$ if and only if $v \leq_\tau w$ and $w \leq_\tau v$ for $v,w \in V$. Since $\leq_\tau$ is a preorder, $\leq_\tau$ defines a partial order on the equivalence classes of $\sim_\tau$ in $V$. 
\end{defi}

\begin{lemma}[{\cite[p.9]{Marcinkowski}}] \label{equivalenceclassesleqtau}
Let $\Gamma=(V,E)$ be a graph $M \subset V$ be an equivalence class of $\sim_\tau$. Then $W_{\Gamma_M}$ is either finite and abelian, or free abelian, or a free group.
\end{lemma}

\begin{defi}[Lower Cone]
Let $Y$ be a set and $\leq$ a relation on $Y$. A subset $X \subset Y$ is called a \textit{lower cone} if for all $t \in X$, $s \in Y$ the relation $s \leq t$ implies that $s \in X$. 
\end{defi}

\begin{lemma}\label{lowerconeunion}
Unions and intersections of lower cones are lower cones. 
\end{lemma}

\begin{proof}
Let $X_i \subset Y$ be lower cones with respect to $\leq$ for $i \in I$. Let $ t \in \bigcup_{i \in I} X_i$ and $s \in Y$ be given such that $s \leq t$. Then there exists $j \in I$ such that $t \in X_j$. Since $X_j$ is a lower cone, $s \in X_j$ and consequently $s \in \bigcup_{i \in I} X_i$. Therefore, $\bigcup_{i \in I} X_i$ is a lower cone. 

Let $t \in \bigcap_{i \in I} X_i$ and $s \in Y$ be such that $s \leq t$. Then $t \in X_i$ for all $i \in I$. Since every $X_i$ is a lower cone, we have that $s \in X_i$ for all $i \in I$ and therefore $s \in \bigcap_{i \in I} X_i$. Therefore, $\bigcap_{i \in I} X_i$ is a lower cone. 
\end{proof} 

\begin{example}
The complement of any minimally chosen subset of vertices $V^\prime \subset V$ whose removal disconnects the graph $\Gamma=(V,E)$ is a lower cone with respect to $\leq_\tau$. In fact, let $C_1, \dots, C_k$ be the vertex sets of the connected components of $\Gamma_{V-V^\prime}$. Then for all $i$ each vertex $v \in C_i$ satisfies $st(v) \subset C_i \cup V^\prime$. But since $V^\prime$ is chosen minimally with respect to the property that its removal disconnects $\Gamma$, the link of each vertex $w \in V^\prime$ contains vertices from at least two distinct connected components $C_i$ and $C_j$. Thus, no vertex $v \in V^\prime$ satisfies $v \leq_\tau w$ for any $w \in V-V^\prime= C_1 \cup \dots \cup C_k$. 
\end{example}

\begin{lemma} \label{maxequivalenceclass}
Let $v \in V$ such that $st(v)=V$ and assume that $G_v$ is finite for all $v$. Then the equivalence class $[v]$ of $\sim_\tau$ is maximal with respect to $\leq_\tau$. 
\end{lemma}

\begin{proof}
Let $w \in V$ be such that $v \leq_\tau w$. Then $V=st(v)=st(w)$. Since $\tau_{v,w}$ is well-defined, $|G_v|=p^k$ and $|G_w|=p^\ell$ for a prime $p$. So $\tau_{w,v}$ is well defined as well. Thus $w \leq_\tau v$ and therefore $w \sim_\tau v$. 
\end{proof}

\begin{lemma}
Let $R_\Gamma$ be a right angled Artin group. Let $v \in V$ such that $st(v)=V$. Then the equivalence class $[v]$ of $\sim_\tau$ is maximal with respect to $\leq_\tau$.
\end{lemma}

\begin{proof}
Let $w \sim_\tau v$ for some $w \neq v$. Since $st(v)=V$, we have that $w \in lk(v)$ and so $v \in lk(w)$. Moreover, since $v \leq_\tau w$, we have $V - \{v\}= lk(v) \subset st(w)$. Thus, $st(w)=V$ for all $w \sim_\tau v$.  Trivially, $lk(z) \subset V=st(v)$ for all $z \in V$ and so any $z \in V$ satisfies $z \leq_\tau v$. 
\end{proof}

\begin{lemma}\label{directtruncatedlowercone}
Let $\Gamma=(V,E)$ be a graph in which no vertex $v \in V$ satisfies $st(v)=V$. If $W_\Gamma$ decomposes as a product of direct truncated subgroups $W_\Gamma \cong W_{\Gamma_{V_1}} \times W_{\Gamma_{V_2}}$, then both $V_1$ and $V_2$ are lower cones with respect to $\leq_\tau$. 
\end{lemma}

\begin{proof}
By the definition of direct truncated subgroups $V$ is the disjoint union $V=V_1 \cup V_2$. Because of the direct product decomposition $W_\Gamma \cong W_{\Gamma_{V_1}} \times W_{\Gamma_{V_2}}$ each generator $v \in V_1$ commutes with each generator $w \in V_2$ and vice versa. If there were $v \in V_1$, $w \in V_2$ such that $lk(v) \subset st(w)$, then $V_2 \subset st(w)$ and so $st(w)=V$ which is a contradiction. So there do not exist any vertices $v \in V_1$, $w \in V_2$ such that $v \leq_\tau w$ which implies that $V_2$ is a lower cone. By symmetry, $V_1$ is a lower cone as well.  
\end{proof}

\begin{defi}
Let $\Gamma=(V,E)$ be a graph and $M \subset V$ be a subset. Define 
\[
L_M = \{ v \in V \ |\ [v,w] \neq e \text{ for all } w \in M \}.
\]
By definition $M \cap L_M = \emptyset$ for all $M$. Equivalently, $L_M=\{  v \in V \ |\ st(v) \subset V-M \}$ and so $L_M= V- (\bigcup_{w \in M} st(w))$. 
\end{defi}

The next lemma is crucial for finding a lower cone in a graph such that the associated graph product of the lower cone decomposes as a non-trivial free product. 

\begin{lemma} \label{equivclassLowercone}
The set $M \cup L_M$ is a lower cone with respect to $\leq_\tau$ for every minimal equivalence class $M \subset V$ of $\sim_\tau$. Moreover, if $G_v$ is finite for all $v \in V$, then $L_M$ is itself a lower cone for any subset $M \subset V$.
\end{lemma}

\begin{proof}
Let $w \in V$. If $w \leq_\tau v$ for $v \in M$, then $w \in M$ by the minimality of $M$ with respect to $\leq_\tau$. Now assume that $w \leq_\tau v$ for $v \in L_M$. If $w$ has finite order, then $w \leq_s v$. Therefore, $st(w) \subset st(v) \subset V-M$ and so $w \in L_M$. Finally, if $w$ has infinite order, then we have $lk(w) \subset st(v) \subset V-M$.  So either $w \in M$ or $st(w) \subset V-M$. Thus, $w \in M \cup L_M$.   
\end{proof}

\begin{lemma}\label{equivclassesLMinRAAG}
Let $M$, $N$ be two equivalence classes such that $W_{\Gamma_N}$ is not a free group of rank $k \geq 2$. If $N \cap L_M \neq \emptyset$, then $N \subset L_M$. 
\end{lemma}

\begin{proof}
Let $a \in M$ and $x \in N \cap L_M$.  Then $x \in L_M$ implies that $ a \notin lk(x)$. Since $M \cap L_M = \emptyset$, it follows that $a \notin st(x)$. Any $y \sim_\tau x$ satisfies $lk(y) \subset st(x)$. So $a \notin lk(y)$. Since $W_{\Gamma_N} \ncong F_k$, the vertices $x$ and $y$ are connected by an edge. Therefore, $y \notin M$. So $st(y) \in V-M$. Since $a \in M$ was chosen arbitrarily and $y \sim_\tau x$ was chosen arbitrarily, it follows that $N \subset L_M$.   
\end{proof}

\begin{lemma}[{\cite[Lemma 3.5]{Marcinkowski}}] \label{lowerconeinvariant}
If $X \subset V$ is a lower cone with respect to  $\leq_\tau$, then $K_X$ is invariant under $\Aut^0(W_\Gamma)$.
\end{lemma}

The following two lemmata follow immediately from Lemma \ref{lowerconeinvariant}.

\begin{lemma}\label{KXcharacteristic}
Let $X \subset V$ be a lower cone with respect to $\leq_\tau$ that is invariant under labelled graph automorphisms. Then $K_X$ is a characteristic subgroup of $W_\Gamma$. 
\qed
\end{lemma}

\begin{lemma}\label{Gen8}
Let $\psi \colon W_\Gamma \to W_\Gamma$ be a labelled graph automorphism. Then $\psi$ preserves the relation $\leq_\tau$. Consequently, $\psi$ preserves equivalence classes of $\sim_\tau$.  
\end{lemma}

\begin{proof}
Clearly, $\psi$ satisfies $|G_v|= |G_{\psi(v)}|$ for all $v \in V$. Moreover, $\psi(st(v)) \subset st(\psi(v))$ for all $v \in V$. By applying $\psi^{-1}$ it follows that $\psi(st(v)) = st(\psi(v))$ for all $v \in V$. Therefore, $\psi$ preserves $\leq_\tau$ and so $\psi$ preserves equivalence classes of $\sim_\tau$. 
\end{proof}

\begin{cor} \label{equivclassfull}
Let $M,N$ be equivalence classes of $\sim_\tau$ and let $\psi \colon W_\Gamma \to W_\Gamma$ be a labelled graph automorphism. Then $\psi(M) \cap N \neq \emptyset$ if and only if $\psi(M)=N$. 
\end{cor}

\begin{lemma}
Let $X \subset V$ be a lower cone with respect to $\leq_\tau$ and $Y$ be the orbit of $X$ under the action of labelled graph automorphisms. Then $K_Y$ is a characteristic subgroup of $W_\Gamma$. 
\end{lemma} 

\begin{proof}
By Lemma \ref{Gen8} the image of any lower cone under a labelled graph automorphism is a lower cone. By Lemma \ref{lowerconeunion} the union of lower cones is a lower cone itself and so the result follows from Lemma \ref{KXcharacteristic}.
\end{proof}

\begin{theorem} \label{PropAutcone}
Let $X \subset V$ be a lower cone with respect to $\leq_\tau$ that is additionally invariant under all labelled graph automorphisms. If $W_{\Gamma_X}$ is a free product of $k \geq 2$ freely indecomposable groups $G_1, \dots, G_k$ such that at least one of the $G_i$ is not $\Z/2$ and at most two are infinite cyclic, then $W_\Gamma$ admits infinitely many linearly independent homogeneous Aut-invariant quasimorphisms. 
\end{theorem}

\begin{proof}
By Lemma \ref{KXcharacteristic} the kernel of the retraction map $R_X \colon W_\Gamma \to W_{\Gamma_X}$ is characteristic. Consequently, by Lemma \ref{Gen6} any Aut-invariant quasimorphism of $W_{\Gamma_X}$ gives rise to an Aut-invariant quasimorphism on $W_\Gamma$. The result follows from Theorem \ref{T3}. 
\end{proof}

\begin{example}
Let $\Gamma$ and $\Lambda$ be the graphs pictured below. Consider the right angled Artin groups $R_\Gamma$ and $R_\Lambda$. Recall that in this case for any two vertices $a,b$ the relation $a \leq_\tau b$, which is by definition equivalent to the existence of the dominated transvection $\tau_{a,b}$, is satisfied if and only if $lk(a) \subset st(b)$. 
\begin{center}
\begin{tikzpicture}[scale=1.0] 

	\draw[fill=black] (0,0) circle(.08);
	\draw[fill=black] (2,1) circle(.08);
 	\draw[fill=black] (2,0) circle(.08);
   	\draw[fill=black] (2,-1) circle(.08);
   	\draw[fill=black] (4,0) circle(.08);
	
	\draw[very thick] (0,0) -- (4,0);
	\draw[very thick] (0,0) -- (2,1) -- (4,0);
	\draw[very thick] (0,0) -- (2,-1) -- (4,0);
	
	\node at (0,0.3) {$v_0$};
	\node at (2,1.3) {$v_1$};
	\node at (2,-0.7) {$v_2$};
	\node at (2,0.3) {$v_3$};
	\node at (4,0.3) {$v_4$};
	\node at (2,-2) {Figure 1: graph $\Gamma$};

	\draw[fill=black] (6,0) circle(.08);
	\draw[fill=black] (8,1) circle(.08);
 	\draw[fill=black] (8,0) circle(.08);
   	\draw[fill=black] (8,-1) circle(.08);
   	\draw[fill=black] (10,0) circle(.08);
   	\draw[fill=black] (12,1) circle(.08);
   	\draw[fill=black] (12,-1) circle(.08);
	
	\draw[very thick] (6,0) -- (10,0);
	\draw[very thick] (10,0) -- (12,1);
	\draw[very thick] (10,0) -- (12,-1);
	\draw[very thick] (6,0) -- (8,1) -- (10,0);
	\draw[very thick] (6,0) -- (8,-1) -- (10,0);
	
	\node at (6,0.3) {$w_0$};
	\node at (8,1.3) {$w_1$};
	\node at (8,-0.7) {$w_2$};
	\node at (8,0.3) {$w_3$};
	\node at (10,0.3) {$w_4$};
	\node at (12,1.3) {$w_5$};
	\node at (12,-0.7) {$w_6$};
	\node at (8,-2) {Figure 2: graph $\Lambda$};
	
\end{tikzpicture} 
\end{center}

In $R_\Gamma$ the sets of vertices $X=\{v_0, v_4\}$ and $Y=\{v_1,v_2,v_3\}$ both form minimal equivalence classes of $\sim_\tau$. Since the cardinality of $X$ and $Y$ differs, both are preserved by any labelled graph automorphism by Lemma \ref{equivclassfull} . Since $R_{\Gamma_X}$ is the free product of two infinite cyclic groups, Proposition \ref{PropAutcone} applies.

In $R_\Lambda$ the vertices $w_0$ and $w_4$ are not equivalent with respect to $\sim_\tau$. In fact, $R_\Lambda$ contains four equivalence classes $A=\{w_0\}$, $B=\{w_1,w_2,w_3\}$, $C=\{w_4\}$ and $D=\{w_5, w_6\}$, where $C$ is the only equivalence class that is not minimal. Let $X=A \cup D$ and $Y= A \cup C \cup D$. Then $X$ and $Y$ are both lower cones that are invariant under labelled graph automorphisms since such automorphism preserve the cardinality of the star of each vertex. However, since $R_{\Lambda_X}$ is the group of rank three, Proposition \ref{PropAutcone} does not apply to $X$. Instead, Proposition \ref{PropAutcone} applies to $Y$ because $R_{C \cup D} \cong \Z \times F_2$ is freely indecomposable since it has non-trivial center. 
\end{example}

\subsection{Aut-invariant quasimorphisms on some classes of right angled Artin groups}

\begin{theorem}\label{RAAG1}
Let $\Gamma$ be a finite graph and $R_\Gamma$ be the right angled Artin group on $\Gamma$. Assume that one of the following two conditions is satisfied
\begin{itemize}
\item There is a minimal equivalence class $M$ of $\sim_\tau$ such that $R_{\Gamma_M} \cong F_2$,
\item No equivalence class $N$ of $\sim_\tau$ satisfies $R_{\Gamma_N} \cong F_k$ for $k \geq 2$. 
\end{itemize}
Then $R_\Gamma$ is either a free abelian group or $R_\Gamma$ admits infinitely many linearly independent homogeneous Aut-invariant quasimorphisms.
\end{theorem}

\begin{proof}
First, assume that $M$ is a minimal equivalence class of  $\sim_\tau$ such that $R_{\Gamma_M} \cong F_2$. Then $M$ is a lower cone of $R_\Gamma$ by Lemma \ref{lowerconeinvariant}. This implies that any unbounded Aut-invariant quasimorphism on $F_2$ gives rise to an unbounded $\Aut^0$-invariant quasimorphism on $R_\Gamma$. 

Let $\sigma$ be a labelled graph automorphism of $\Gamma$ and let $\psi$ be an unbounded Aut-invariant quasimorphism on $R_{\Gamma_M}$. According to Lemma \ref{equivclassfull} either $\sigma(M)=M$, in which case $\sigma$ descends to an automorphism of $R_{\Gamma_M}$, or $\sigma(M) \cap M = \emptyset$. Since the labelled graph automorphisms $\{\sigma_i\}$ form a set of coset representative for $\Aut^0(R_{\Gamma})$ in $\Aut(R_{\Gamma})$ according to Lemma \ref{MarciLemma} , the quasimorphism $\widehat{\psi \circ p}$ defined for $g \in R_{\Gamma}$ by $\widehat{\psi \circ p}(g)= \sum_i \psi(p(\sigma_i(g)))$ is invariant under the whole automorphism group $\Aut(R_{\Gamma})$ by Lemma \ref{finindexinvariant}.

Thus, for all $g \in R_{\Gamma_M} \leq R_\Gamma$ it holds that
\begin{align*}
\widehat{\psi \circ p}(g)= \sum_i \psi(p(\sigma_i(g))) = |J| \psi(g),
\end{align*}
where $J$ is the subset of labelled graph automorphisms satisfying $\sigma_i(M)=M$. Clearly, $|J| \geq 1$. So the Aut-invariant quasmorphism $\widehat{\psi \circ p}$ restricted to the subgroup $R_{\Gamma_M} \leq R_\Gamma$ is just a linear multiple of $\psi$. Then the result follows from Theorem \ref{T1} , which in this case is due to \cite[Thm. 2]{BraMarc}. 

Second, assume that no equivalence class $N$ of $\sim_\tau$ satisfies $R_{\Gamma_N} \cong F_k$ for $k \geq 2$. Then every equivalence class $N$ satisfies $R_{\Gamma_N} \cong \Z^k$ for some $k \geq 1$ by Lemma \ref{equivalenceclassesleqtau}. Let $M$ be a minimal equivalence class. By Lemma \ref{equivclassLowercone} $M \cup L_M$ is a lower cone. If $L_M = \emptyset$, then $M$ commutes with all other elements. By Lemma \ref{maxequivalenceclass} $M$ is maximal and so $V-M$ is a lower cone. Proceed with choosing a minimal equivalence class in $V-M$ and iterate the construction. This process will either yield a nontrivial $N$ such that $L_N$ is non-trivial or implies that $R_\Gamma$ is a free abelian group itself. So we may now assume that $L_M$ is non-trivial and we consider the lower cone $M \cup L_M$. 

According to Lemma \ref{equivclassesLMinRAAG} any equivalence class $N$ of $\sim_\tau$ with $N \cap L_M \neq \emptyset$ is fully contained in $L_M$. Let $N$ be a minimal equivalence class in $L_M$. Then $M \cup N$ is a lower cone. Consequently, $p \colon R_\Gamma \to R_{\Gamma_M} \ast R_{\Gamma_{N}}$ is an $\Aut^0$-equivariant projection. Moreover, $R_{\Gamma_M} \ast R_{\Gamma_{N}} \cong \Z^\ell \ast \Z^k$ for some $k, \ell \geq1$. Let $\psi \colon R_{\Gamma_M} \ast R_{\Gamma_{N}} \to \Z$ be an Aut-invariant quasimorphism. Let $\sigma$ be labelled graph automorphism of $\Gamma$. According to Lemma \ref{equivclassfull} either $\sigma(M)=M$, in which case $\sigma$ descends to an automorphism of $R_{\Gamma_M}$, or $\sigma(M) \cap M = \emptyset$. Since the labelled graph automorphisms $\{\sigma_i\}$ form a set of coset representative for $\Aut^0(R_{\Gamma})$ in $\Aut(R_{\Gamma})$ according to Lemma \ref{MarciLemma} , the quasimorphism $\widehat{\psi \circ p}$ defined for $g \in R_{\Gamma}$ by $\widehat{\psi \circ p}(g)= \sum_i \psi(p(\sigma_i(g)))$ is invariant under the whole automorphism group $\Aut(R_{\Gamma})$ by Lemma \ref{finindexinvariant}.

Thus for any Aut-invariant quasimorphism $\psi$ on $R_{\Gamma_M} \ast R_{\Gamma_{N}}$ we have for $g \in R_{\Gamma_M} \ast R_{\Gamma_{N}} \leq R_\Gamma$ that
\begin{align*}
\widehat{\psi \circ p}(g)= \sum_i \psi(p(\sigma_i(g))) = |J| \psi(g),
\end{align*}
where $J$ is the subset of labelled graph automorphisms satisfying $\{ \sigma(M), \sigma(N)\}=\{M,N\}$. Clearly, $|J| \geq 1$. So restricted to the subgroup $R_{\Gamma_M} \ast R_{\Gamma_{N}} \leq R_\Gamma$ the Aut-invariant quasmorphism is just a linear multiple of $\psi$. Then the statement follows from Theorem \ref{T1}. 
\end{proof}

\begin{proof}[Proof of Thm. \ref{RAAG2}]
Since any equivalence class $M$ of $R_\Gamma$ consisting of at least two vertices admits non-trivial transvections by definition, every equivalence class of $\sim_\tau$ consists of a single vertex. Since there are two vertices that do not commute, $R_\Gamma$ cannot be free abelian. Thus, the statement follows from Proposition \ref{RAAG1}.
\end{proof}

\subsection{Freely indecomposable graph products}

\begin{lemma}\label{finiteorder}
Let $A \ast B$ be a free product of two non-trivial groups. If $x \in A \ast B$ has finite order then $x$ is a conjugate of a letter belonging to $A$ or $B$. 
\end{lemma}

\begin{proof}
Let $x \in A \ast B$ be an element that has finite order and is not conjugate to a letter and such that the reduced form $w=c_1 \dots c_k$ of $x$ has minimal length $k$ among all elements of finite order in $A \ast B$ that are not conjugate to a letter. If $c_1$ and $c_k$ were letters from different factors, the order of $x$ would be infinite. So $c_1 $ and $c_k$ belong to the same factor. Let $c$ be the letter representing the product of $c_k$ with $c_1$ in that factor. Then the reduced form of $c_1^{-1} x c_1$ is given by $ c_2 \dots c_{k-1} a$, which has length $k-1$ contradicting the minimality of $x$. Consequently, the result follows.
\end{proof}

\begin{lemma}\label{freelyindecomposable}
Let $\Gamma$ be a connected graph of primary groups. Then $W_\Gamma$ is freely indecomposable. 
\end{lemma}

\begin{proof}
Assume $W_\Gamma \cong G_1 \ast G_2$. By Lemma \ref{finiteorder} every element of finite order is conjugate to a letter of $G_1$ or $G_2$. Let $v \in \Gamma$ be a vertex generating a primary group $G_v$. Without loss of generality we can assume that $v$ is conjugate to a non-trivial letter of $G_1$. That is, $v$ can be written as a reduced word $ag_1a^{-1}$, where $g_1 \in G_1$ and $a$ is a reduced word. 

Let $w$ be adjacent to $v$. Then $w$ is conjugate to a non-trivial letter by Lemma \ref{finiteorder}. We will prove by contradiction that $w$ is a conjugate of a letter of $G_1$ as well. Assume that $w$ was conjugate to a letter of $G_2$. Then $w$ can be written as a reduced word as $bg_2 b^{-1}$ where $g_2 \in G_2$. Let $c$ be the reduced form of $a^{-1} b$. Since $v$ and $w$ are adjacent, they commute. Then 
\begin{align*}
1=a^{-1}[v,w]a = a^{-1}[ag_1a^{-1},bg_2 b^{-1}]a & = g_1a^{-1} b g_2 b^{-1} a g_1^{-1} a^{-1} b g_2^{-1} b^{-1} a \\
& = g_1 c g_2 c^{-1} g_1^{-1} c g_2 c^{-1}.
\end{align*}
If $c$ is the empty word, this yields a contradiction since $g_1$ and $g_2$ are non-trivial letters belonging to different factors of the free product and so their commutator is non-trivial. Therefore, $c$ is non-trivial. If the last letter of $c$ belongs to $G_2$, then we define $c_0$ to be the reduced word obtained by omitting this last letter of $c$. Otherwise we set $c_0=c$. Then $cg_2c^{-1}=c_0 g_2 c_0^{-1}$ where the latter expression is reduced. Consequently, the expression $c_0 g_2^{-1} c_0^{-1}$ is reduced and represents $c g_2^{-1} c^{-1}$. We compute
\[
1=a^{-1}[v,w]a= g_1 c g_2 c^{-1} g_1^{-1} c g_2 c^{-1} = g_1  c_0 g_2 c_0^{-1}  g_1^{-1}  c_0 g_2^{-1} c_0^{-1}.
\]
Again, this would yield a contradiction if $c_0$ represented the identity. However, if $c_0$ was to begin with a letter from $G_2$, the expression $g_1  c_0 g_2 c_0^{-1}  g_1^{-1}  c_0 g_2^{-1} c_0^{-1}$ would be reduced and therefore be non-trivial as well. Therefore, $c_0$ begins with a non-trivial letter $x \in G_1$ and its reduced form has length $\geq 2$. Let $c_1$ be the non-trivial reduced word obtained by omitting the first letter of $c_0$. Let $y$ be the letter from $G_1$ representing the product of $g_1$ and $x$ or empty if $g_1 x =1$. Then the product of reduced words $y c_1 g_2 c_1^{-1} g_1^{-1} c_1  g_2^{-1} c_1^{-1} x^{-1}$ is a non-trivial reduced word itself and represents the same element as $g_1  c_0 g_2 c_0^{-1}  g_1^{-1}  c_0 g_2^{-1} c_0^{-1}$. This is a contradiction. 

Consequently, $w$ is a conjugate of a letter belonging to $G_1$. By induction it follows that all vertex groups belong to the conjugacy class of $G_1$ in $W$ since $\Gamma$ is connected. Since the vertex groups generate $W_\Gamma$, it follows that $W_\Gamma$ is completely contained in the conjugacy class of $G_1$. Then it lies in the kernel of the projection onto the factor $G_2$ and so $G_2$ is the trivial group. It follows that $W_\Gamma$ is freely indecomposable. 
\end{proof}

\subsection{Aut-invariant quasimorphisms on graph products of finite abelian groups}

Recall that a direct truncated subgroup of a graph product $W$ on a graph $\Gamma=(V,E)$ is spanned by a subset of the vertex set $V^\prime \subset V$ such that $W$ decomposes as a cartesian product of the graph products on the subsets $V^\prime$ and $V-V^\prime$. Moreover, recall that by definition a graph product of finite groups $W$ is finite if and only if the underlying graph is complete since $W$ has a non-trivial free product as a subgroup otherwise.

\begin{T} \label{T5}
Let $\Gamma=(V,E)$ be a finite graph. Let $\{G_v \}_{v \in V}$ a family of finite abelian groups and let $W(\Gamma, \{G_v \}_{v \in V})$ be their graph product. Let $Z_k$ be the $k$-fold free product of groups of order two $Z_k=\Z/2 \ast \dots \ast \Z/2$. Assume that $W(\Gamma, \{G_v \}_{v \in V})$ does not decompose as a product $G_1 \times \dots \times G_\ell$ for $\ell \geq 1$ where each $G_i$ is a direct truncated subgroup that is isomorphic to $Z_k$ for some $k  \geq 1$ or finite abelian.  Then $W(\Gamma, \{G_v \}_{v \in V})$ admits infinitely many linearly independent homogeneous Aut-invariant quasimorphisms.  
\end{T}

\begin{proof}
We can replace $\Gamma$ and $\{ G_v \}_{v \in V}$ with a graph $\Gamma^\prime$ and a collection of primary abelian groups $\{ G_v^\prime \}_{v \in V}$ such that $W(\Gamma, \{G_v \}_{v \in V}) \cong W(\Gamma^\prime, \{G_v^\prime \}_{v \in V})=: W_{\Gamma^\prime}$. If $\Gamma^\prime$ has more than one connected component, then $W_{\Gamma^\prime}$ is the free product of the graph products associated to the connected components of $\Gamma^\prime$. By Lemma \ref{freelyindecomposable} each of those connected components is itself freely indecomposable and by assumption $W(\Gamma, \{G_v \}_{v \in V})$ is not a free product of groups of order two. So Theorem \ref{T3} applies and the result follows. 

Therefore, we may assume that $\Gamma^\prime$ is a connected graph. Moreover, recall that by Lemma \ref{direct trunc subgroup} if $Z_k$ is a direct truncated subgroup of  $W(\Gamma^\prime, \{G_v^\prime \}_{v \in V})$, then $Z_k$ is a direct truncated subgroup in $W(\Gamma, \{G_v \}_{v \in V})$. We will outline an iterative procedure proving the following claim.

\begin{claim}
If $W(\Gamma, \{G_v \}_{v \in V})$ is not finite, there exists a lower cone $L$ in $\Gamma^\prime$ and $n \geq 2$ such that  $W_{\Gamma^\prime_L}$ is a free product of $W_1, \dots, W_n$, the groups generated by the connected components of $L$, such that  $W_1, \dots, W_n$ are not infinite cyclic and are not all equal to $\Z/2$.
\end{claim}

If $st(v)=V$ for some $v \in V$, then the equivalence class $[v] \subset V$ is maximal with respect to $\leq_\tau$ by Lemma \ref{maxequivalenceclass}. Thus, $V-[v]$ is a lower cone with respect to $\leq_\tau$ and $W_{\Gamma^\prime} \cong W_{\Gamma^\prime_{[v]}} \times W_{\Gamma^\prime-[v]}$ where $\Gamma^\prime-[v]$ is the subgraph of $\Gamma^\prime$ spanned by all vertices $v \in V-[v]$. According to Lemma \ref{equivalenceclassesleqtau} the group $W_{\Gamma^\prime_{[v]}}$ is finite and abelian. We then proceed by considering the graph $\Gamma^\prime - [v]$. Therefore, we will assume in the following that no $v \in V$ satisfies $st(v)=V$. 

From now onwards let $M$ be an equivalence class which is minimal with respect to $\leq _\tau$. Again, Lemma \ref{equivalenceclassesleqtau} implies that $M$ generates a finite abelian group $W_{\Gamma^\prime_M}$ in $W_{\Gamma^\prime}$, since all $G_v^\prime$ are finite. Moreover, $st(v)=st(w)$ for all $v,w \in M$. We consider the set of vertices $M \cup L_M$ which form a lower cone with respect to $\leq_\tau$ according to Lemma \ref{equivclassLowercone}. The set $L_M$ is non-empty since $st(v) \neq V$ for $v \in M$. Since no vertex $v \in M$ shares an edge with a vertex $w \in L_M$, the group $W_{M \cup L_M}$ decomposes as a free product $W_M \ast W_{L_M}$. 

Since no free factor can be the infinite cyclic group this proves the claim unless $W_M \cong Z/2$ and all connected components of $L_M$ consist of single vertices with vertex groups equal to $\Z/2$. Then $M= \{x\}$ and $L_M=\{y_1, \dots , y_\ell \}$ for some $\ell \geq 1$ and both are lower cones with respect to $\leq_\tau$ according to Lemma \ref{equivclassLowercone}. In fact, since all equivalence classes generate finite abelian groups, this implies that each vertex $y_i$ is its own equivalence class and therefore minimal since $L_M$ is a lower cone. All vertices in $V-(M \cup L_M)$ commute with $x$ by definition of $L_M$. If there was $z \in V-(M \cup L_M)$ such that $[z,y_i] \neq 1$ for some $i$, then $x,z \in L_{ \{y_i \} }$. In this case $\{y_i \} \cup L_{ \{y_i\} }$ would be a lower cone for which $W_{ \{y_i \} \cup L_{ \{y_i\} } }$ decomposes as a free product of two freely indecomposable factors which are not all equal to $\Z/2$ since one connected component contains an edge. 

Otherwise, we are in the situation where $x$ and all $y_i$ commute with all other vertices. Then $M \cup L_M$ generates a direct truncated subgroup in $\Gamma^\prime$ that is isomorphic to $Z_k$. By Lemma \ref{direct trunc subgroup} $Z_k$ comes from a direct truncated subgroup spanned by the same vertex set in $W(\Gamma, \{G_v \}_{v \in V})$. By Lemma \ref{directtruncatedlowercone} the complement $V-(M \cup L_M)$ is a lower cone and we restart the iterative procedure using the graph $\Gamma_{V-(M \cup L_M)}$. If $W_\Gamma$ is not finite, this process either terminates with a lower cone as specified in the claim or yields a decomposition of $W_\Gamma$ into direct truncated subgroups all of which are of the form $Z_k$ for some $k \geq 2$. However, the latter is impossible by assumption. This proves the claim. \\

It follows from the claim that by Lemma \ref{lowerconeinvariant} and Lemma \ref{Aut0-projection} there exists an $\Aut^0$-equivariant map $p \colon W_{\Gamma^\prime} \to W_a \ast W_b$ for some $a,b \in \{1, \dots, n \}$ such that $W_a \ast W_b \ncong D_\infty$. Let $A$ in $W_a$ and $B$ in $W_b$ be minimal equivalence classes. We distinguish two cases.

First, consider the case where $\Gamma_A$ and $\Gamma_B$ are not both equal to $\Z/2$. These two equivalence classes are lower cones themselves, which means that the projection $q \colon W_{\Gamma^\prime} \to W_{\Gamma_A} \ast W_{\Gamma_B}$ is $\Aut^0$-equivariant. Let $\psi \colon W_{\Gamma_A} \ast W_{\Gamma_B} \to \R$ be any unbounded Aut-invariant quasimorphism which always exists according to Theorem \ref{T1}. Then $\psi \circ p$ is an unbounded $\Aut^0$-invariant quasimorphism on $W_{\Gamma^\prime}$. Since the labelled graph automorphisms $\{\sigma_i\}$ form a set of coset representative for $\Aut^0(W_{\Gamma^\prime})$ in $\Aut(W_{\Gamma^\prime})$ according to Lemma \ref{MarciLemma} , the quasimorphism $\widehat{\psi \circ q}$ defined for $g \in W_{\Gamma^\prime}$ by $\widehat{\psi \circ q}(g)= \sum_i \psi(q(\sigma_i(g)))$ is invariant under the whole automorphism group $\Aut(W_{\Gamma^\prime})$ by Lemma \ref{finindexinvariant}. It remains to check that it is unbounded. 

Any labelled graph automorphism $\sigma \in \Aut(W_{\Gamma^\prime})$ satisfies by Lemma \ref{equivclassfull}  
\begin{align*}
\sigma(A) \cap A, \sigma(A) \cap B, \sigma(B) \cap A, \sigma(B) \cap A  \in \{A,B \}.
\end{align*}
Thus, any labelled graph automorphism either descends to an automorphism of $W_{\Gamma^\prime_A} \ast W_{\Gamma^\prime_B}$ or satisfies that $q(\sigma(w))$ is a single letter in $W_{\Gamma^\prime_A} \ast W_{\Gamma^\prime_B}$ for all words $w \in W_{\Gamma^\prime_A} \ast W_{\Gamma^\prime_B}$. Let $J \leq I$ be the subset of labelled graph automorphisms in $W_{\Gamma^\prime}$ that descend to $W_{\Gamma^\prime_A} \ast W_{\Gamma^\prime_B}$. Then $|J| \geq 1$ and for all $g \in W_{\Gamma_M^\prime} \ast W_{\Gamma_N^\prime}$ considered as a subgroup of $W_{\Gamma^\prime}$ one calculates  
\begin{align*}
\widehat{\psi \circ q}(g)= \sum_i \psi(q(\sigma_i(g))) = |J| \psi(g).
\end{align*}
Therefore, the Aut-invariant quasimorphism $\widehat{\psi \circ p}$ is unbounded on $W_{\Gamma^\prime}$. Finally, linear independence of quasimorphisms constructed in this way follows from linear independence of the quasimorphisms in Theorem \ref{T1}.

Second, consider the case where all non-trivial minimal equivalence classes in $A \in W_a$ and $B \in W_b$ just consist of single vertices $A = \{v_1 \} $ and $B= \{v_2 \}$ with both vertex groups equal to $\Z/2$. Since $W_a$ and $W_b$ are freely indecomposable, both of them are connected. Since it holds that $W_a \ast W_b \ncong D_\infty$, there exists a second vertex in at least one of them. So without loss of generality we assume there is $v_3 \in W_b$ that is adjacent to $v_2$. Replace $W_a$ by just $G_{v_1}$ since $\{v_1\}$ is a lower cone itself and denote the resulting projection map $W_\Gamma \to G_{v_1} \ast W_b$ again by $p$ for simplicity. Let $\sigma$ be a labelled graph automorphism of $W_{\Gamma^\prime}$. We will proceed similarly to the first case of the proof to show that the symmetrisation of the homogenisation of our counting quasimorphisms is unbounded, by showing that it is unbounded when we restrict to suitable subgroups of $W_{\Gamma}$. For this we distinguish the cases where $G_{v_3}$ is not equal to $\Z/2$ and where it is. 

If $G_{v_3} \neq \Z/2$ then $\sigma(v_3)=v_1$ and $\sigma(v_1)=v_3$ are impossible. So if we have $p(\sigma(v_1)) \neq 0$ and $p(\sigma(v_3)) \neq 0$, then either $\sigma(v_1), \sigma(v_3)$ both belong to $W_{b}$ or $\sigma(v_1) = v_1$ and $\sigma(v_3) \in W_{b}$. So restricted to the subgroup $G_{v_1} \ast G_{v_3}$ either a labelled graph automorphism $\sigma$ preserves the $W_b$-code or the result is a letter. Consequently, we calculate using Lemma \ref{naturalitycode} for the last equality that
\begin{align*}
(\widehat{\bar{f}^{W_b}_z \circ p})|_{G_{v_1} \ast G_{v_3}}= \sum_i \bar{f}^{W_b}_z(p(\sigma_i |_{G_{v_1} \ast G_{v_3}})) = |J| \cdot (\bar{f}^{W_b}_z)|_{G_{v_1} \ast G_{v_3}} =|J| \cdot \bar{f}^{G_{v_3}}_z,
\end{align*}
where $J$ is the subset of labelled graph automorphisms satisfying $\sigma(v_1)$ and $\sigma(v_3) \in W_b$. Since $J$ contains the identity, $|J| \geq 1$. Therefore, $\widehat{\bar{f}^{W_b}_z \circ p}$ is unbounded by Proposition \ref{freeprod no Z prop} which proves the result in this case.  

However, if $G_{v_3} = \Z/2$, then $p(\sigma(G_{v_1} \ast (G_{v_2} \times G_{v_3}))) \leq D_\infty \leq W_a \ast W_b$ if $\sigma(v_2)=v_1$ or $\sigma(v_3)=v_1$ since $v_2$ and $v_3$ are adjacent. So $p(\sigma(G_{v_1} \ast (G_{v_2} \times G_{v_3}))) \in \{ D_\infty \leq W_a \ast W_b, W_a, W_b \}$ unless $\sigma(v_1)=v_1$ and $\sigma(\{v_2,v_3\}) \subset W_b$. Similarly to the previous case we calculate 
\begin{align*}
(\widehat{\bar{f}^{W_b}_z \circ p})|_{G_{v_1} \ast (G_{v_2} \times G_{v_3})} &= \sum_i \bar{f}^{W_b}_z(p(\sigma_i |_{G_{v_1} \ast (G_{v_2} \times G_{v_3})})) = |J| \cdot (\bar{f}^{W_b}_z)|_{G_{v_1} \ast (G_{v_2} \times G_{v_3})} \\
&=|J| \cdot \bar{f}^{G_{v_2} \times G_{v_3}}_z,
\end{align*}
where we again used Lemma \ref{naturalitycode} for the last equality and $J$ denotes the subset of labelled graph automorphisms satisfying $\sigma(v_1)=v_1$ and $\sigma(\{v_2,v_3\}) \subset W_b$. As before, $|J| \geq 1$, so $\widehat{\bar{f}^{W_b}_z \circ p}$ is unbounded by Proposition \ref{freeprod no Z prop}. This concludes the last case and proves the result of the theorem.
\end{proof}

\begin{cor}
Let $\Gamma=(V,E)$ be a finite graph. Let $\{G_v \}_{v \in V}$ be a family of finite abelian groups and let $W(\Gamma, \{G_v \}_{v \in V})$ be their graph product. If no vertex group is equal to $\Z/2$ then either $W(\Gamma, \{G_v \}_{v \in V})$ is finite or it admits unbounded homogeneous Aut-invariant quasimorphisms. 
\end{cor}

\begin{proof}
If the $k$-fold free product of groups of order two $Z_k=\Z/2 \ast \dots \ast \Z/2$ is a truncated subgroup $H$ of $W(\Gamma, \{G_v \}_{v \in V})$ on the vertex set $V^\prime \subset V$, then for $v^\prime \in V^\prime$ all vertex groups $G_{v^\prime}$ are groups of order two since $H$ is the free product of the connected components of $\Gamma_{V^\prime}$. Consequently, if there is $G_v \neq \Z/2$ for $v \in V$, then $W(\Gamma, \{G_v \}_{v \in V})$ does not decompose as a product of direct truncated subgroups each of which is isomorphic to some $Z_k$ for $k \geq 2$. The result follows from Theorem \ref{T5}.
\end{proof}

\begin{remark}
This is a super-strong version of the so-called $bq$-dichotomy studied in \cite{BGKM}, where instead of boundedness of a group $G$ one has finiteness and instead of any unbounded quasimorphism on $G$ one has unbounded quasimorphisms that are invariant under all automorphisms of $G$.
\end{remark}

Moreover, we obtain the following corollary of which Theorem \ref{introcorollary} is a special case.

\begin{cor}\label{lkcorollary}
$\Gamma=(V,E)$ be a finite graph that is not complete and let $W_\Gamma$ be a graph product of finite abelian groups on $\Gamma$. If there are no two vertices $v,w \in V$ such that $G_v=G_w=\Z/2$ and $lk(v)=lk(w)$ then $W_\Gamma$ admits infinitely many linearly independent homogeneous Aut-invariant quasimorphisms.  
\end{cor}

\begin{proof}
If the $k$-fold free product $Z_k = \Z/2 \ast \dots \ast \Z/2$ is a direct truncated subgroup of $W_\Gamma$, then all vertices generating the free factors of $Z_k$ have the same link. So by assumption $Z_k$ can never be a direct truncated subgroup of $W_\Gamma$ for any $k \geq 2$. Moreover, since $\Gamma$ is not complete, there are two vertices $v,w \in V$ that do not commute. Therefore, $W_\Gamma$ is infinite and the result follows from Theorem \ref{T5}.
\end{proof}

\section{Examples}

Let us now consider some families of examples of graphs and analyse which graph products of finite groups and which right angled Artin groups on these graphs admit unbounded Aut-invariant quasimorphisms. For a graph $\Gamma$ we will denote by $W_\Gamma$ the graph product of a family of finite groups on $\Gamma$ and by $R_\Gamma$ the right angled Artin group on $\Gamma$. Implicitly, we replace $\Gamma$ by $\Gamma^\prime$ so that $W_{\Gamma^\prime} \cong W_\Gamma$ is a graph product where each vertex group is primary if this was not already the case for $W_\Gamma$. Further, recall that groups of the form $(D_\infty)^k \times A$ do not admit any unbounded Aut-invariant quasimorphisms for all $k \geq 0$ whenever $A$ is abelian.   

\begin{example}\label{ngon}
Let $\Gamma_n$ be the graph of the regular $n$-gon. If $n=3$, then $W_{\Gamma_3} \cong G_{v_1} \times G_{v_2} \times G_{v_3}$ and $R_{\Gamma_3} \cong \Z^3$ are always abelian. 

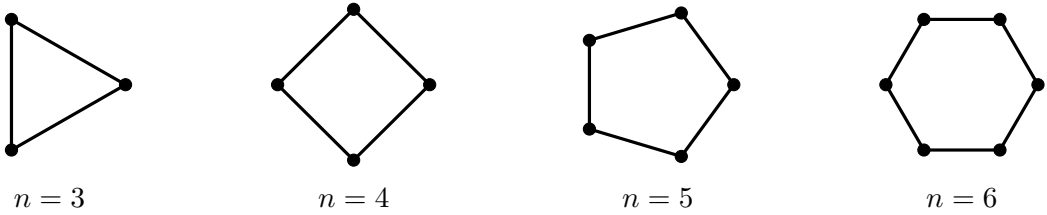
\begin{center}
\begin{tikzpicture}[scale=1.0] 
   
	\draw[fill=black] (1,0) circle(.08);
	\draw[fill=black] (-0.5,0.866) circle(.08);
	\draw[fill=black] (-0.5,-0.866) circle(.08);
	
	\draw[very thick] (1,0) -- (-0.5,0.866);
	\draw[very thick] (1,0) -- (-0.5,-0.866);
	\draw[very thick] (-0.5,-0.866) -- (-0.5,0.866);
	
	\draw[fill=black] (5,0) circle(.08);
	\draw[fill=black] (3,0) circle(.08);
	\draw[fill=black] (4,1) circle(.08);
	\draw[fill=black] (4,-1) circle(.08);
	
	\draw[very thick] (5,0) -- (4,1);
	\draw[very thick] (5,0) -- (4,-1);
	\draw[very thick] (3,0) -- (4,-1);
	\draw[very thick] (3,0) -- (4,1);
	
	\draw[fill=black] (9,0) circle(.08);
	\draw[fill=black] (8.309,0.951) circle(.08);
	\draw[fill=black] (8.309,-0.951) circle(.08);
	\draw[fill=black] (7.101,0.588) circle(.08);
	\draw[fill=black] (7.101,-0.588) circle(.08);
	
	\draw[very thick] (9,0) -- (8.309,0.951);
	\draw[very thick] (9,0) -- (8.309,-0.951);
	\draw[very thick] (8.309,0.951) -- (7.101,0.588);
	\draw[very thick] (8.309,-0.951) -- (7.101,-0.588);
	\draw[very thick] (7.101,0.588) -- (7.101,-0.588);	
	
	\draw[fill=black] (13,0) circle(.08);
	\draw[fill=black] (12.5,0.866) circle(.08);
	\draw[fill=black] (12.5,-0.866) circle(.08);
	\draw[fill=black] (11.5,0.866) circle(.08);
	\draw[fill=black] (11.5,-0.866) circle(.08);
	\draw[fill=black] (11,0) circle(.08);
	
	\draw[very thick] (13,0) -- (12.5,0.866) ;
	\draw[very thick] (13,0) -- (12.5,-0.866);
	\draw[very thick] (12.5,0.866) -- (11.5,0.866);
	\draw[very thick] (12.5,-0.866) -- (11.5,-0.866);
	\draw[very thick] (11.5,0.866) -- (11,0);	
	\draw[very thick] (11.5,-0.866) -- (11,0);
	
	\node at (0,-1.5) {$n=3$};
	\node at (4,-1.5) {$n=4$};
	\node at (8,-1.5) {$n=5$};
	\node at (12,-1.5) {$n=6$};
	
\end{tikzpicture} 
\end{center}
For $n \geq 5$ the links of no two vertices are equal and so $W_{\Gamma_n}$ always admits infinitely many linearly independent Aut-invariant quasimorphisms according to Corollary \ref{lkcorollary}. Moreover, the condition $lk(v) \subset st(w)$ is never satisfied for any distinct vertices $v,w$ and so $R_{\Gamma_n}$ does not admit any transvections. Therefore, it follows from Corollary \ref{RAAG2} that $R_{\Gamma_n}$ admits infinitely many linearly independent Aut-invariant quasimorphisms. 

For $n=4$ it follows from Theorem \ref{T5} that $W_{\Gamma_4}$ has infinitely many linearly independent homogeneous Aut-invariant quasimorphisms except for the single case where all vertex groups are $\Z/2$ and so $W_{\Gamma_4} \cong D_\infty \times D_\infty$. Furthermore, for $R_{\Gamma_4} \cong F_2 \times F_2$ opposite vertices belong to the same equivalence class with respect to $\sim_\tau$. Both of these classes are minimal and it follows from Proposition \ref{RAAG1} that $R_{\Gamma_4}$ admits infinitely many linearly independent homogeneous Aut-invariant quasimorphisms.
\end{example}

\begin{example} \label{ncube}
Let $\Gamma_n$ be the graph given by the 1-skeleton of the $n$-cube. Clearly, for $n=1$ we have that $W_{\Gamma_1} \cong G_{v_1} \times G_{v_2}$ and $R_{\Gamma_1} \cong \Z^2$ are abelian. The case $n=2$ is the case of the regular $4$-gon discussed in Example \ref{ngon} above. 
\begin{center}
\begin{tikzpicture}[scale=1.0] 
%1-cube   
	
	\draw[fill=black] (1.5,0.8) circle(.06);
	\draw[fill=black] (1.5,-0.8) circle(.06);

	\draw[very thick] (1.5,0.8) -- (1.5,-0.8);
	
%2-cube
	
	\draw[fill=black] (5,0.8) circle(.06);
	\draw[fill=black] (5,-0.8) circle(.06);
	\draw[fill=black] (3.4,0.8) circle(.06);
	\draw[fill=black] (3.4,-0.8) circle(.06);
	
	\draw[very thick] (5,0.8) -- (5,-0.8) ;
	\draw[very thick] (5,0.8) -- (3.4,0.8);
	\draw[very thick] (5,-0.8)  -- (3.4,-0.8);
	\draw[very thick] (3.4,0.8) -- (3.4,-0.8);
	
%3-cube
	
	\draw[fill=black] (8.5,0.8) circle(.06);
	\draw[fill=black] (8.5,-0.8) circle(.06);
	\draw[fill=black] (6.9,0.8) circle(.06);
	\draw[fill=black] (6.9,-0.8) circle(.06);
	
	\draw[very thick] (8.5,0.8) -- (8.5,-0.8) ;
	\draw[very thick] (8.5,0.8) -- (6.9,0.8);
	\draw[very thick] (8.5,-0.8)  -- (6.9,-0.8);
	\draw[very thick] (6.9,0.8) -- (6.9,-0.8);
	
	\draw[fill=black] (9.1,1.4) circle(.06);
	\draw[fill=black] (9.1,-0.2) circle(.06);
	\draw[fill=black] (7.5,1.4) circle(.06);
	\draw[fill=black] (7.5,-0.2) circle(.06);
	
	\draw[very thick] (9.1,1.4) -- (9.1,-0.2) ;
	\draw[very thick] (9.1,1.4) -- (7.5,1.4);
	\draw[very thick] (9.1,-0.2)  -- (7.5,-0.2);
	\draw[very thick] (7.5,1.4) -- (7.5,-0.2);
	
	\draw[very thick] (9.1,1.4) -- (8.5,0.8) ;
	\draw[very thick] (6.9,0.8) -- (7.5,1.4);
	\draw[very thick] (9.1,-0.2)  -- (8.5,-0.8);
	\draw[very thick] (6.9,-0.8) -- (7.5,-0.2);

%outer cube
	
	\draw[very thick] (11,-1.6) -- (11,1.6) ;
	\draw[very thick] (14.2,-1.6) -- (14.2,1.6);
	\draw[very thick] (14.2,1.6) -- (11,1.6) ;
	\draw[very thick] (14.2,-1.6) -- (11,-1.6);
	
	\draw[very thick] (11.7,-0.6) -- (11.7,2.6) ;
	\draw[very thick] (14.9,-0.6) -- (14.9,2.6);
	\draw[very thick] (14.9,2.6) -- (11.7,2.6) ;
	\draw[very thick] (14.9,-0.6) -- (11.7,-0.6);
	
	\draw[very thick] (11,-1.6) -- (11.7,-0.6) ;
	\draw[very thick] (14.2,-1.6) -- (14.9,-0.6);
	\draw[very thick] (11,1.6) -- (11.7,2.6) ;
	\draw[very thick] (14.2,1.6) -- (14.9,2.6);
	
%inner cube	
	
	\draw[very thick] (12.3,-0.3) -- (13.3,-0.3) ;
	\draw[very thick] (12.3,0.7) -- (13.3,0.7) ;
	\draw[very thick] (12.3,-0.3) -- (12.3,0.7) ;
	\draw[very thick] (13.3,-0.3) -- (13.3,0.7) ;
	
	\draw[very thick] (12.533,0.033) -- (13.533,0.033);
	\draw[very thick] (12.533,1.033) -- (13.533,1.033);
	\draw[very thick] (12.533,0.033) -- (12.533,1.033);
	\draw[very thick] (13.533,0.033) -- (13.533,1.033);
	
	\draw[very thick] (12.3,-0.3) --  (12.533,0.033);
	\draw[very thick] (12.3,0.7) -- (12.533,1.033) ;
	\draw[very thick] (13.3,-0.3) -- (13.533,0.033) ;
	\draw[very thick] (13.3,0.7) -- (13.533,1.033) ;
	
%fourth dim
	
	\draw[very thick] (13.533,1.033) -- (14.9,2.6);
	\draw[very thick] (13.533,0.033) -- (14.9,-0.6);
	\draw[very thick] (13.3,0.7) -- (14.2,1.6);
	\draw[very thick] (13.3,-0.3) -- (14.2,-1.6);
	
    \draw[very thick] (11,-1.6) -- (12.3,-0.3);
	\draw[very thick] (11,1.6) -- (12.3,0.7);
	\draw[very thick] (11.7,-0.6) -- (12.533,0.033);
	\draw[very thick] (11.7,2.6) -- (12.533,1.033);
	
%nodes
	\draw[fill=black] (11,-1.6) circle(.06);
	\draw[fill=black] (11,1.6) circle(.06);
	\draw[fill=black] (14.2,-1.6) circle(.06);
	\draw[fill=black] (14.2,1.6) circle(.06);
	\draw[fill=black] (11.7,-0.6) circle(.06);
	\draw[fill=black] (11.7,2.6) circle(.06);
	\draw[fill=black] (14.9,-0.6) circle(.06);
	\draw[fill=black] (14.9,2.6) circle(.06);
	
	\draw[fill=black] (12.3,-0.3) circle(.06);
	\draw[fill=black] (13.3,-0.3) circle(.06);
	\draw[fill=black] (12.3,0.7) circle(.06);
	\draw[fill=black] (13.3,0.7) circle(.06);
	\draw[fill=black] (12.533,0.033) circle(.06);
	\draw[fill=black] (13.533,0.033) circle(.06);
	\draw[fill=black] (12.533,1.033) circle(.06);
	\draw[fill=black] (13.533,1.033) circle(.06);
	
%n's	
	\node at (1.5,-1.5) {$n=1$};
	\node at (4.2,-1.5) {$n=2$};
	\node at (7.7,-1.5) {$n=3$};
	\node at (12.8,-2.3) {$n=4$};
\end{tikzpicture} 
\end{center}
Let $n \geq 3$. The links of no two vertices are equal and the condition $lk(v) \subset st(w)$ is never satisfied for any two distinct $v,w \in V_n$. Consequently, it follows from Corollary \ref{lkcorollary} and Corollary \ref{RAAG2} that $W_{\Gamma_n}$ and $R_{\Gamma_n}$ always admit infinitely many linearly independent Aut-invariant quasimorphisms.
\end{example}

\begin{example}[Platonic solids]
The graph given by the tetrahedron $T$ is the complete graph on $4$ vertices and therefore $W_T \cong G_{v_1} \times G_{v_2} \times G_{v_3} \times G_{v_4}$ and $R_T \cong \Z^4$ are abelian. The cube $C$ is the case $n=3$ of Example \ref{ncube} above. For the icosahedron $I$ and the dodecahedron $D$ the conditions $lk(v)=lk(w)$ and $lk(v) \subset st(w)$ are both never satisfied for distinct vertices $v,w$. It follows from Corollary \ref{lkcorollary} and Corollary \ref{RAAG2} that $W_I$, $W_D$, $R_I$ and $R_D$ all admit infinitely many linearly independent Aut-invariant quasimorphisms. Let $O$ be the graph of the octahedron pictured below. 
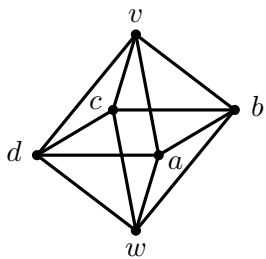
\begin{center}
\begin{tikzpicture}[scale=1.0] 	
%3-cube

	\draw[fill=black] (8,-0.8) circle(.06);
	\draw[fill=black] (6.4,-0.8) circle(.06);
	\draw[very thick] (8,-0.8)  -- (6.4,-0.8);
	
	\draw[fill=black] (9,-0.2) circle(.06);
	\draw[fill=black] (7.4,-0.2) circle(.06);
	
	\draw[very thick] (9,-0.2)  -- (7.4,-0.2);
	\draw[very thick] (9,-0.2)  -- (8,-0.8);
	\draw[very thick] (6.4,-0.8) -- (7.4,-0.2);
	
	\draw[fill=black] (7.7,0.8) circle(.06);
	
	\draw[very thick] (9,-0.2)  -- (7.7,0.8);
	\draw[very thick] (7.4,-0.2)  -- (7.7,0.8);
	\draw[very thick] (8,-0.8)  -- (7.7,0.8);
	\draw[very thick] (6.4,-0.8)  -- (7.7,0.8);
	
	\draw[fill=black] (7.7,-1.8) circle(.06);
	
	\draw[very thick] (9,-0.2)  -- (7.7,-1.8);
	\draw[very thick] (7.4,-0.2)  -- (7.7,-1.8);
	\draw[very thick] (8,-0.8)  -- (7.7,-1.8);
	\draw[very thick] (6.4,-0.8)  -- (7.7,-1.8);
	
	\node at (7.7,-2.05) {$w$};
	\node at (7.7,1.05) {$v$};
	
	\node at (9.3,-0.15) {$b$};
	\node at (6.1,-0.75) {$d$};
	\node at (7.17,-0.1) {$c$};
	\node at (8.22,-0.9) {$a$};

\end{tikzpicture} 
\end{center}
In $R_O$ there are the three equivalence classes $\{a,c\}$, $\{b,d\}$ and $\{v,w\}$, all of which are minimal with respect to $\sim_\tau$ and generate $F_2$. So it follows from Proposition \ref{RAAG1} that $R_O$ admits infinitely many linearly independent Aut-invariant quasimorphisms. In $W_O$ any two opposite vertices have the same link and $W_O \cong (G_v \ast G_w) \times (G_a \ast G_c) \times (G_b \ast G_d)$ as a product of direct truncated subgroups. It follows from Theorem \ref{T5} that $W_O$ admits infinitely many linearly independent Aut-invariant quasimorphisms unless all vertex groups are equal to $\Z/2$ in which case $W_O$ is isomorphic to $(D_\infty)^3$.

\end{example}

In the above examples we can never find lower cones invariant under labelled graph automorphisms that are not the whole graph. Let us see in some less symmetrical graphs how we can extract Aut-invariant quasimorphisms by more elementary means from lower cones and Proposition \ref{PropAutcone} without appealing to the full strength of Theorem \ref{T5}.

\begin{example} \label{nline}
Let $A_n$ denote the graph given by the standard triangulation of the interval $[0,n]$ with all integers in $[0,n]$ as vertices and edges of length one in between them. 
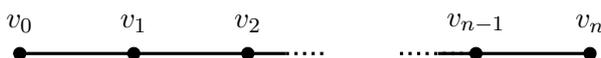
\begin{center}
\begin{tikzpicture}[scale=1.0]   
	
	\draw[fill=black] (0,0) circle(.08);
	\draw[fill=black] (1.5,0) circle(.08);
	\draw[fill=black] (3,0) circle(.08);
	\draw[fill=black] (6,0) circle(.08);
	\draw[fill=black] (7.5,0) circle(.08);

	\draw[very thick] (0,0) -- (1.5,0);
	\draw[very thick] (1.5,0) -- (3.5,0);
	\draw[very thick] (5.5,0) -- (7.5,0);	
	
	\draw[very thick,dotted] (3.5,0) -- (4,0);
	\draw[very thick,dotted] (5,0) -- (6,0);
	
	\node at (0,0.4) {$v_0$};
	\node at (1.5,0.4) {$v_1$};
	\node at (3,0.4) {$v_2$};
	\node at (6,0.4) {$v_{n-1}$};
	\node at (7.5,0.4) {$v_n$};
	
\end{tikzpicture}
\end{center}

For $n=0,1$ the groups $W_{A_n}$ and $R_{A_n}$ are all abelian and so none of them admit unbounded Aut-invariant quasimorphisms.  

For $n \geq 2$ the set $\{v_0, v_n \}$ is a lower cone invariant under labelled graph automorphisms. Then it follows directly from Proposition \ref{PropAutcone} that $R_{A_n}$ and $W_{A_n}$ admit infinitely many linearly independent homogeneous Aut-invariant quasimorphisms except for $W_{A_n}$ in the case where $G_{v_0} = G_{v_n} = \Z/2$. So assume that $G_{v_0} = G_{v_n} = \Z/2$ for the rest of this example. For $n \geq 4$ the set $\{v_0, v_1, v_{n-1}, v_n \}$ forms a lower cone as well and Proposition \ref{PropAutcone} again yields the existence of infinitely many linearly independent homogeneous Aut-invariant quasimorphisms. For $n=2$ it holds that $W_{A_2} \cong G_{v_1} \times D_\infty$  and so $W_{A_2}$ does not admit any unbounded quasimorphisms since $D_\infty$ does not according to Example \ref{Dinfty} and $G_{v_2}$ is finite. Only to settle the case $n=3$ we note that no two distinct vertices have the same link and so Corollary \ref{lkcorollary} yields the existence of infinitely many linearly independent Aut-invariant quasimorphisms. 
\end{example}

\begin{remark}
The particular choice of lower cone is usually not unique. For example, for $n \geq 8$ the set $\{ v_2, v_3, v_5, v_6 \}$ is a lower cone that intersects trivially with the lower cone  $\{v_0, v_1, v_{n-1}, v_n \}$ chosen in the argument before. Different choices of lower cones yield entirely different Aut-invariant quasimorphisms.  
\end{remark}

\begin{example} \label{ntwoendedline}
Let $B_n$ denote the graph which is $A_{n-2}$ with two additional vertices attached to the vertex $n-2$. So $B_n$ has $n+1$ vertices. Then $B_2$ is the graph $A_2$ of Example \ref{nline} above. 
\begin{center}
\begin{tikzpicture}[scale=1.0]   
	
	\draw[fill=black] (0,0) circle(.08);
	\draw[fill=black] (1.5,0) circle(.08);
	\draw[fill=black] (3,0) circle(.08);
	\draw[fill=black] (6,0) circle(.08);
	\draw[fill=black] (7.5,0) circle(.08);
	\draw[fill=black] (9,1) circle(.08);
	\draw[fill=black] (9,-1) circle(.08);

	\draw[very thick] (0,0) -- (1.5,0);
	\draw[very thick] (1.5,0) -- (3.5,0);
	\draw[very thick] (5.5,0) -- (7.5,0);	
	\draw[very thick] (7.5,0) -- (9,1);
	\draw[very thick] (7.5,0) -- (9,-1);
	
	\draw[very thick,dotted] (3.5,0) -- (4,0);
	\draw[very thick,dotted] (5,0) -- (6,0);
	
	\node at (0,0.4) {$v_0$};
	\node at (1.5,0.4) {$v_1$};
	\node at (3,0.4) {$v_2$};
	\node at (6,0.4) {$v_{n-3}$};
	\node at (7.4,0.4) {$v_{n-2}$};
	\node at (9,1.4) {$v_{n-1}$};
	\node at (9,-0.6) {$v_{n}$};
	
\end{tikzpicture}
\end{center}
For $n \geq 4$ the set $\{v_{n-2}, v_{n-1}, v_n \}$ forms a lower cone of $W_{B_n}$ and $R_{B_n}$ respectively. Moreover, $\{v_0\}$ is a minimal equivalence class with respect to $\sim_\tau$. So $\{v_0 ,  v_{n-2}, v_{n-1}, v_n \}$ forms a lower cone $L$ of $W_{B_n}$ and $R_{B_n}$. $L$ is invariant under labelled graph automorphisms and the truncated subgroup generated by $L$ is isomorphic to $G_{v_0} \ast \big ( G_{v_{n-2}} \times ( G_{v_{n-1}} \ast G_{v_n}) \big )$. Since the group $G_{v_{n-2}} \times ( G_{v_{n-1}} \ast G_{v_n})$ has non-trivial center, it is freely indecomposable. So Proposition \ref{PropAutcone} implies that  $W_{B_n}$ and $R_{B_n}$ admit infinitely many linearly independent homogeneous Aut-invariant quasimorphisms for all $n \geq 4$.

For $n=3$ the vertices $\{v_0, v_2, v_3\}$ form a lower cone. By Proposition \ref{PropAutcone} the graph product $W_{B_3} \cong G_{v_1} \times (G_{v_0} \ast G_{v_2} \ast G_{v_3})$ admits infinitely many linearly independent quasimorphisms if $G_{v_0}$, $G_{v_2}$, $G_{v_3}$ are not all equal to $\Z/2$.
\end{example}

\begin{remark}
With our methods we cannot settle the case $n=3$ for $B_n$ fully since we do not know whether $F_3$ or $\Z/2 \ast \Z/2 \ast \Z/2$ admit unbouned Aut-invariant quasimorphisms. 
\end{remark}

\section{Aut-invariant stable commutator length}

For a group $G$ denote the commutator length on $[G,G]$ by $\cl_G$. That is, $\cl_G(x)$ is the minimal number of commutators needed to express $x \in [G,G]$. The \textit{stable commutator length} $\scl_G$ is defined for $x \in [G,G]$ as the limit $\scl_G(x)=\lim_n \frac{\cl(x^n)}{n}$ and shares a deep relationship with quasimorphisms von $G$ via the so-called Bavard duality (see \cite[Thm. 2.70]{Calegari}). We now define an invariant analogue of these notions according to \cite[p.7]{KK} in the case where $\hat{G}$ is a group that contains $G$ as a normal subgroup. 

\begin{defi}
Let $G \in \hat{G}$ be a normal subgroup. Let $[\hat{G},G]$ be the subgroup of $G$ generated by elements of the form $[F,g]$ and their inverses where $F \in \hat{G}$, $g \in G$. For $x \in [\hat{G},G]$ define the $\hat{G}$\textit{-invariant commutator length} $\cl_{\hat{G},G}(x)$ to be the minimal number of elements of the above form $[F,g]$ and their inverses needed to express $x$. The \textit{stable} $\hat{G}$\textit{-invariant commutator length} $\scl_{\hat{G},G}(x)$ is defined as the limit $\scl_{\hat{G},G}(x)=\lim_n \frac{\cl_{\hat{G},G}(x^n)}{n}$ for $x \in [\hat{G},G]$. 
\end{defi}

If $G$ is any group with trivial center, then it can be identified with its group of inner automorphisms $\Inn(G)$, which is a normal subgroup of the automorphism group $\Aut(G)$. Thus for $\hat{G}=\Aut(G)$ the above definition then yields the so-called \textit{Aut-invariant stable commutator length} $\scl_{\Aut(G),G}$ on $G$, which we will denote by $\scl_{\Aut}$ to simplify notation. 

For $\hat{G}=\Aut(G)$ the following lemma is proven in \cite[Lemma 2.1]{KK}.

\begin{lemma} \label{Kawasaki Lemma}
Let $G$ be a group with trivial center so that $G= \Inn(G)$. Let $\phi$ be a homogeneous Aut-invariant quasimorphism on $G$. Then any $x \in [\Aut(G),G]\leq G$ satisfies 
\begin{align*}
\scl_{\Aut}(x) \geq \frac{1}{2}\frac{|\phi(x)|}{D(\phi)}. 
\end{align*} \qed
\end{lemma} 

According to \cite[Theorem 1.3]{KK} $\hat{G}$-invariant quasimorphisms satisfy an analogue of the Bavard duality theorem if $[\hat{G},G] =G$. For a graph product of finitely generated abelian groups the invariant $\scl_{\Aut}$ is always defined on a subgroup of finite index. In fact, it is often defined on the whole graph product.  

\begin{theorem}\label{finiteindexgraphproduct}
Let $\Gamma=(V,E)$ be a graph and let $\{G_v\}_{v \in V}$ be a collection of finitely generated abelian groups. Let $G=W(\Gamma, \{G_v\}_{v \in V})$ be their graph product. Assume that no vertex $v$ satisfying $st(v)=V$ has a vertex group satisfying $|G_v|= \infty$. Let $\hat{G}=\Aut(G)$. Then $[\hat{G},G]$ is a subgroup of finite index in $G$. Moreover, if no vertex $v \in V$ satisfies $st(v)=V$ and no element inside any vertex group $G_v$ has infinite order or order equal to a power of $2$, then $[\hat{G},G]= G$. 
\end{theorem}

\begin{proof}
Since $W(\Gamma, \{G_v\}_{v \in V})$ is isomorphic to a graph product of groups where every vertex group is either primary or infinite cyclic, we assume without loss of generality that this is the case for $\Gamma=(V,E)$ and $\{G_v\}_{v \in V}$ and write $W_\Gamma=W(\Gamma, \{G_v\}_{v \in V})$. Moreover, if no element inside any vertex group has infinite order or order equal to a power of $2$, then all vertex groups can be assumed to be primary for primes $\neq 2$. 

By Corollary \ref{Green2} the center of $W_\Gamma$ is generated by the generators $v \in V$ that satisfy $st(v)=V$. Since all vertex groups $G_v$ of those vertices are finite abelian groups, the center of $W_\Gamma$ is a finite abelian group and we may assume up to passing to a finite index subgroup of $W_\Gamma$, that $W_\Gamma$ has trivial center.

The abelianisation $W_\Gamma / [W_\Gamma, W_\Gamma]$ is obtained by taking the quotient over all $[G_v, G_w]$ for $v,w \in V$. Consequently, $W_\Gamma / [W_\Gamma, W_\Gamma] =  \bigoplus_{v \in V} G_v$. If no vertex group $G_v$ has infinite order this implies that $[\hat{W_\Gamma},W_\Gamma]$ has finite index, since $[W_\Gamma, W_\Gamma ] \leq [\hat{W_\Gamma},W_\Gamma]$. Otherwise let $v \in V$ be such that $G_v $ is infinite cyclic and by abuse of notation denote a generator of $G_v$ by $v$ as well. Since $st(v) \neq V$, then conjugation by $v$ defines a nontrivial element inside $\Aut(W_\Gamma)$. Let $r \in \Aut(W_\Gamma)$ be the factor automorphism of $W_\Gamma$ that inverts $v$. In $\Aut(W_\Gamma)$ we compute $[r,v^{-1}]=r(v^{-1}) \cdot v = v^2$. Thus, the subgroup generated by $v^2$ is contained in $[\hat{W_\Gamma},W_\Gamma]$ as well. Let $H$ be the group $\bigoplus_{v \in V^\prime} G_v  \oplus \bigoplus_{v \in V \setminus V^\prime} G_v/2$ where $V^\prime = \{ v \in V : \hspace{2mm} | G_v | < \infty \}$ . $H$ is a finite group and maps surjectively onto $W_\Gamma / [\hat{W_\Gamma},W_\Gamma]$ , which implies that $[\hat{W_\Gamma},W_\Gamma]$ has finite index in $W_\Gamma$. 

For the second part of the statement assume that all vertex groups are primary for primes $\neq 2$. Let $k \geq 1$ and let $p \neq 2$ be a prime. Since $p \neq 2$, the map $m \colon \Z/p^k \to \Z/p^k$ induced by multiplication by $2$ is an automorphism. Let $G_v$ be a vertex group such that $G_v = \Z/p^k$ and regard $m$ as the corresponding factor automorphism of $W_\Gamma$. Then $[m,v]= m(v) \cdot v^{-1} = v^2 \cdot v^{-1} = v$. Therefore, $G_v$ is a subgroup of $[\hat{W_\Gamma},W_\Gamma]$. Since $W_\Gamma$ is generated by the vertex groups, it follows that $[\hat{W_\Gamma},W_\Gamma]=W_\Gamma$.  
\end{proof}

Let $Z_k$ be the $k$-fold free product of groups of order two $Z_k=\Z/2 \ast \dots \ast \Z/2$.

\begin{theorem}\label{sclpositivegraphprod}
Let $\Gamma=(V,E)$ be a finite graph in which $st(v) \neq V$ for all $v \in V$. Then there always exist elements of positive Aut-invariant stable commutator length in the following groups:
\begin{enumerate}
\item $W(\Gamma, \{G_v \}_{v \in V})$, the graph product of a family of finite abelian groups $\{G_v \}_{v \in V}$, if the group $W(\Gamma, \{G_v \}_{v \in V})$ does not decompose as a product $G_1 \times \dots \times G_\ell$ for $\ell \geq 1$ where each $G_i$ is a direct truncated subgroup that is isomorphic to some $Z_k$ for $k \geq 1$ or finite abelian.
\item $R_\Gamma$, the right angled Artin group on $\Gamma$, if one of the following two conditions is satisfied:
\begin{itemize}
\item There is a minimal equivalence class $M$ of $\sim_\tau$ such that $R_{\Gamma_M} \cong F_2$,
\item No equivalence class $N$ of $\sim_\tau$ satisfies $R_{\Gamma_N} \cong F_k$ for $k \geq 2$. 
\end{itemize}
\end{enumerate}
\end{theorem}

\begin{proof}
By assumption all groups have trivial center according to Corollary \ref{Green2}. There exist unbounded homogeneous Aut-invariant quasimorphisms on $W_\Gamma=W(\Gamma, \{G_v \}_{v \in V})$ and $R_\Gamma$ according to Theorem \ref{T5} and Proposition \ref{RAAG1}. Since $[\widehat{W_\Gamma},W_\Gamma] \leq W_\Gamma$ and $[\widehat{R_\Gamma},R_\Gamma] \leq R_\Gamma$ have finite index by Proposition \ref{finiteindexgraphproduct} these quasimorphisms are unbounded on $[\widehat{W_\Gamma},W_\Gamma]$ and $[\widehat{R_\Gamma},R_\Gamma]$ respectively. Therefore, Lemma \ref{Kawasaki Lemma} implies the result.
\end{proof}

%Let $R_\Gamma$ be the right angled Artin group on $\Gamma$. Define $V^\prime = \{ v \in V \ |\ st(v) =V\}$. By Corollary \ref{Green2} the center of $R_\Gamma$ is $R_{\Gamma_{V^\prime}}$. Moreover, $R_\Gamma$ decomposes as a product of direct truncated subgroups $R_\Gamma \cong R_{\Gamma_{V^\prime}} \times R_{\Gamma_{V^{ \prime \prime}}}$ where $V^{\prime \prime} = V - V^\prime$. By Proposition \ref{RAAG1} there exists an unbounded homogeneous Aut-invariant quasimorphisms $\psi$ on $R_{\Gamma_{V^{ \prime \prime}}}$. The quasimorphism $\psi$ is unbounded on $[\widehat{R_{\Gamma_{V^{ \prime \prime}}}},R_{\Gamma_{V^{ \prime \prime}}}]$ since $[\widehat{R_{\Gamma_{V^{ \prime \prime}}}},R_{\Gamma_{V^{ \prime \prime}}}]$ has finite index in $R_{\Gamma_{V^{ \prime \prime}}}$ by  Proposition \ref{finiteindexgraphproduct}. Thus, Lemma \ref{Kawasaki Lemma} implies that there are elements of positive Aut-invariant stable commutator length in 
%Could generalise above Prop. to the case of non-complete graphs. But need to adjust definition where we use that for graph products center is always a direct subgroup and prove lemma that cl_Aut is Lipschitz + well def.  

%Moreover, $\psi \circ p \colon R_\Gamma \to \R$ where $p \colon R_\Gamma \to R_\Gamma /R_{\Gamma_{V^\prime}}$ is the quotient projection is Aut-invariant on $\R_\Gamma$ since the center is a characteristic subgroup.

\section*{Acknowledgements}

I would like to thank Jarek K\k{e}dra and Benjamin Martin for their continued support and for their helpful comments. This work was partly funded by the Leverhulme Trust Research Project Grant ${\text{RPG-2017-159}}$.

\vspace{8mm}
\noindent
BASTIEN KARLHOFER, UNIVERSITY OF ABERDEEN, UK.\\
\noindent 
EMAIL: r01bdk17@abdn.ac.uk

%\end{thebibliography}
\end{document}